\documentclass[12pt]{article}

\usepackage{amsmath,amsthm,amssymb}
\usepackage{color}
\usepackage[noadjust,sort]{cite}

\usepackage[colorlinks=true,citecolor=black,linkcolor=black,urlcolor=blue]{hyperref}
\usepackage{mathtools}

\usepackage{ulem}

\normalsize

\theoremstyle{plain}
\newtheorem{theorem}{Theorem}
\newtheorem{corollary}[theorem]{Corollary}
\newtheorem{lemma}[theorem]{Lemma}

\theoremstyle{definition}
\newtheorem*{remark*}{remark}

\def\semiabs#1{\abs{#1}_\pi}

\def\brho{\rho_{\mathrm{big}}}
\def\srho{\rho_{\mathrm{small}}}
\def\imba{b}
\def\imbavec{\boldsymbol{\imba}}
\def\imbamax{\infnorm{\imbavec}}
\def\lognd{\log\dfrac{2n}{\varDelta}}
\def\fl{\varphi}
\def\cmin{c_{\mathrm{min}}}

\def\abs#1{\lvert#1\rvert} \let\card=\abs
\def\Abs#1{\bigl\lvert#1\bigr\rvert} 
\def\st{\mathrel{:}}

\def\norm#1{\mathopen\|#1\mathclose\|}
\def\onenorm#1{\norm{#1}_1}
\def\twonorm#1{\norm{#1}_2}
\def\infnorm#1{\norm{#1}_\infty}
\def\frobnorm#1{\norm{#1}_{\mathrm{F}}}
\def\trans{^{\mathrm{T}}}
\def\nperp{n_{\scriptscriptstyle\perp}}

\def\nfrac#1#2{{\textstyle\frac{#1}{#2}}}
\def\dfrac#1#2{\lower0.15ex\hbox{\large$\textstyle\frac{#1}{#2}$}}
\def\({\bigl(}
\def\){\bigr)}

\let\eps=\varepsilon

\def\calS{\mathcal{S}}

\def\X{\boldsymbol{X}}
\def\Y{\boldsymbol{Y}}

\def\xvec{{\boldsymbol{x}}}
\def\rvec{{\boldsymbol{r}}}
\def\pvec{{\boldsymbol{p}}}
\def\imbavec{{\boldsymbol{b}}}

\def\yvec{{\boldsymbol{y}}}

\def\vvec{{\boldsymbol{v}}}

\def\thetavec{{\boldsymbol{\theta}}}
\def\phivec{{\boldsymbol{\phi}}}
\def\thetapvec{{\boldsymbol{\theta'}}}

\def\ljk{\lambda_{jk}}
\def\lkj{\lambda_{kj}}
\def\lamlam{\ljk\lkj}
\def\lamdiff{\ljk-\lkj}

\def\ellmin{\ell_{\mathrm{min}}}

\def\fre{f_{\mathrm{re}}}
\def\fim{f_{\mathrm{im}}}

\def\E{\operatorname{\mathbb{E}}}
\def\Var{\operatorname{Var}}

\def\indeg{\operatorname{indeg}}
\def\outdeg{\operatorname{outdeg}}
\def\sp{\operatorname{span}}
\def\Cov{\operatorname{Cov}}
\def\diag{\operatorname{diag}}
\def\tr{\operatorname{tr}}

\def\Reals{\mathbb{R}}
\def\Complexes{\mathbb{C}}
\def\Integers{\mathbb{Z}}
\def\Rmodpi{\Reals/\mkern-1mu\pi}

\def\nicebreak{\vskip 0pt plus 50pt\penalty-300\vskip 0pt plus -50pt }

\def\rem{\zeta}
\def\leq{\leqslant} \def\geq{\geqslant} 
\def\le{\leqslant} \def\ge{\geqslant} 
\allowdisplaybreaks

\title{Asymptotic enumeration of orientations of a graph
as a function of the out-degree sequence\thanks
{This research is supported by the Australian Research Council, Discovery Project DP140101519.
The first author's research is also supported by Australian Research  Council  Discovery Early Career Researcher Award DE200101045. }}

\author{ 
Mikhail Isaev\\
\small School of Mathematics\\[-0.8ex]
\small Monash University\\[-0.8ex]
\small Clayton, VIC 3800, Australia\\
\small Moscow Institute of Physics and Technology\\[-0.8ex]
\small Dolgoprudny, 141700, Russian Federation\\
\small\tt mikhail.isaev@monash.edu
\and
Tejas Iyer\\
\small Department of Mathematics\\[-0.8ex]
\small University of Birmingham\\[-0.8ex]
\small Birmingham, UK\\
\small \tt TXI790@student.bham.ac.uk\\
\and
Brendan D. McKay\\
\small Research School of Computer Science\\[-0.8ex]
\small Australian National University\\[-0.8ex]
\small Canberra, ACT 2601, Australia\\
\small\tt brendan.mckay@anu.edu.au\\
}

\begin{document}

\maketitle
\begin{abstract}
We prove an asymptotic formula for the number of orientations with
given out-degree (score) sequence for a graph~$G$.
The graph $G$ is assumed to have average degrees at least
$n^{1/3 + \eps}$ for some $\eps > 0$, and to have strong mixing properties,
while the maximum imbalance (out-degree minus in-degree) of the orientation should be not too large.
Our enumeration results have applications to the study of subdigraph occurrences in random orientations with given imbalance sequence.
As one step of our calculation, we obtain new bounds for the maximum likelihood estimators for the Bradley-Terry model of paired comparisons.
\end{abstract}

\nicebreak
\section{Introduction}\label{s:intro}

Let $G$ be an undirected simple graph with vertices $\{1,2,\ldots,n\}$.
An \textit{orientation} of $G$ is an assignment of one of the two possible directions to each edge, thereby making an oriented graph~$\vec G$.  The \textit{imbalance} (sometimes
called \textit{excess}) of a vertex~$v\in V(\vec G)$ is $\imba_v:=\outdeg(v)-\indeg(v)$,
and the \textit{imbalance sequence} of $\vec G$
is $\imbavec=\imbavec(\vec G):=(\imba_1,\ldots,\imba_n)$.
If $\imbavec(\vec G)=\boldsymbol{0}$, then $\vec G$ is called an \textit{Eulerian
orientation} of~$G$.

Our primary aim in this paper is to find the asymptotic  number of orientations
of $G$ with given imbalance sequence.  In solving this enumeration problem, we will apply the saddle point method to a suitable generating function, using Cauchy's Theorem while following the general framework outlined in \cite{Mother}.
In the process, we will use results from the \textit{theory of paired comparisons}, uncovering an interesting link between mathematical statistics and enumerative combinatorics.

In order to apply the saddle point method to enumerate the number of orientations, we will use the standard parameters in the \textit{Bradley-Terry model of paired comparisons}. This model was  first studied by Zermelo in 1929 \cite{Zermelo1929}, and independently by Bradley and Terry~\cite{Bradley1952}, Ford~\cite{Ford1957}, Jech~\cite{Jech1983} and many others. See, for example, Hunter \cite{hunter2004} for a general treatment.  Contestants in a competition carried out by
pairwise comparisons are assumed to have ``merits'' $\rvec=(r_1,\ldots,r_n)$ such that 
contestant $j$ defeats contestant $k$ with probability
\begin{equation}\label{lambdadef}
   \ljk = \ljk(\rvec) := \frac{r_j}{r_j+r_k}.
\end{equation}
Note that $\ljk+\lkj=1$; i.e., ties are not allowed.
The statistical problem is then to estimate the merits from the scores
(the number of comparisons won by each contestant), after which the
merits can be taken as a measure of the strength of each contestant. 

Each of the above authors noted that the maximum likelihood
estimate of the merits given the scores is (up to multiplication by a
constant factor, since only the ratios matter) the solution of the
``balance equations''
\begin{equation}\label{betaequations}
  \sum_{k:jk\in G} \frac{r_j-r_k}{r_j+r_k} =  \sum_{k:jk\in G} (\lamdiff) = \imba_j,
   \quad  1\le j\le n.
\end{equation}
Zermelo~\cite{Zermelo1929} proved that~\eqref{betaequations} has
a unique solution if the digraph defined by the results of each comparison
is strongly connected. We generalise this in Theorem~\ref{t:existence},
using the fact, earlier noticed by Joe~\cite{Joe},
that~\eqref{betaequations} corresponds to the point maximising a
certain entropy.
As a result of equation~\eqref{betaequations}, the values $\{r_j\}$ are the radii of circles
whose direct product passes through the saddle point of a generating function in
$n$-dimensional complex space; see Section \ref{s:enumeration}.

If we orient each edge $jk$ independently towards $k$
with probability $\ljk$ and towards $j$ with probability $\lkj$, then,
as we will prove in Lemma~\ref{uniformlemma},
the probability of a particular orientation depends only on its imbalance
sequence.
Because of this, it makes sense to choose
$\rvec=(r_1,\ldots,r_n)$ so that the expected imbalances in the induced
orientation equal some sequence $\imbavec$ of interest.

This gives the equations~\eqref{betaequations}.
Note that if $\rvec$ satisfies~\eqref{betaequations}, then so does
$ c \rvec$ for any constant~$c>0$.
In the case of Eulerian orientations, a solution is
$\rvec=(1,\ldots,1)$, which gives $\ljk=\frac12$
for all $jk\in G$.

A special case of our problem is enumeration of \textit{tournaments} with given scores.
Some of the first results go back to Spencer in 1974 \cite{Spencer1974}, who gave an estimate of the number of tournaments with a given imbalance sequence.
More precise results were given in \cite{McKay1990} and \cite{McKayWang} 
based on the complex-analytic approach.
This technique was applied in \cite{GaoMcKayWang2000} to asymptotically enumerate the number of tournaments containing a given small digraph. 
 The method was further generalised  in \cite{isaev, misha} to calculate the number of Eulerian orientations for a large class of dense graphs with strong mixing properties.  In this paper we extend  all of the aforementioned results allowing much sparser graphs and much more variation in the imbalances of vertices.
 
  Note that counting orientations  with a given imbalance sequence of a bipartite graph corresponds to counting its subgraphs with fixed degree sequence (take all edges which go into one of the parts). 
  Equivalently, we can count $0$--$1$ matrices with given margins where some set of entries are forced to be $0$.
     This question goes back to Read \cite{Read1958} in 1958, who derived a formula for  the number of 
     3-regular bipartite graphs. For more recent  asymptotic results, see, for example, \cite{BarvHart2, CGM2008, Greenhill, bipdeg} and references therein. Our formula applied to the bipartite case significantly improves known results for this enumeration problem as well.

The \textit{Cheeger constant} (or isoperimetric number) of  a graph $G$, denoted by $h(G)$,
is defined as follows.
\begin{equation*}
	h(G):=\min
	\biggl\{ \frac{|\partial_G \,U|}{|U|}  \st  U\subset V(G), 1\le |U|\le\dfrac12 |V(G)| \biggr\},
\end{equation*}
where $\partial_G \,U$ is the set of edges of $G$ with one end in $U$ and one end in $V(G) \setminus U$.
The number $h(G)$ is a  discrete analogue of the Cheeger isoperimetric constant in the theory of
Riemannian manifolds and it has many interesting interpretations (for more detailed
information see, for example, \cite{Mohar1989} and the references therein).

Let $I$ denote the identity matrix, and let $J$ denote the matrix with every entry~1; in
each case of order~$n$.
Define the symmetric positive-semidefinite matrix $L=L(G,\imbavec)$ by
\begin{equation}\label{Ldef}
     \xvec\trans\!L\xvec = 2\sum_{jk\in G} \lamlam (x_j-x_k)^2,
\end{equation}
for $\xvec=(x_1,\ldots,x_n)\trans\in\Reals^n$, and further define
\begin{equation}\label{afdefs}
\begin{split}
     A &:= \dfrac{\varDelta}{n} J + L, \\
     f_3(\xvec) &:= -\dfrac43\, \sum_{jk\in G} \lamlam
          (\lamdiff)(x_j-x_k)^3, \\
     f_4(\xvec) &:= 
          \dfrac23 \sum_{jk\in G} \lamlam
          (1-6\lamlam)(x_j-x_k)^4, \\
    f_6(\xvec) &:= - \dfrac4{45} \sum_{jk\in G} \lamlam
             (1 - 30\lamlam + 120\ljk^2\lkj^2)
              (x_j-x_k)^6, \\
    \X &:= ~\parbox[top]{21em}{an $n$-dimensional normally distributed
  random variable
    with density $\pi^{-n/2}\abs{A}^{1/2}e^{-\xvec\trans\!A\xvec}$,}\\
   \psi(G,\imbavec) &:= \E f_4(\X) + \E f_6(\X)-\dfrac12\Var f_3(\X)
            +\dfrac12\Var f_4(\X),\\
  P(G,\imbavec) &:= \frac{ \prod_{j=1}^n r_j^{\outdeg(j)}}
                                      { \prod_{jk\in G} \,(r_j+r_k) },
\end{split}
\end{equation}
where  $\E Z$ and $\Var Z$  stand for  the expectation
and  the variance of a random variable $Z$.

In the following theorem,
a pair $(G, \imbavec)$ stands for a sequence of graphs and imbalance sequences $(G(n), \imbavec(n))$ parametrised by a positive integer  $n$. 
Statements involving $n$ and $\eps$ hold
if $n$ is sufficiently large and $\eps$ is sufficiently small. 
Throughout the paper,  the asymptotic
notations $o(\,), O(\,), \Omega(\,)$ have their usual meaning.

\nicebreak
\begin{theorem}\label{t:bigtheorem}
  Let $G$ be a graph with $n$ vertices and  maximum degree $\varDelta$. Let $\imbavec$ be the imbalance sequence for some orientation of~$G$. Assume the following hold as $n \rightarrow \infty$.
\begin{itemize}\itemsep=0pt
  \item[A1.] $n^{1/3+\eps}\le \varDelta\le n-1$ for some constant $\eps>0$.
  \item[A2.]  $h(G)\ge\gamma\varDelta$, for some constant $\gamma>0$.
  \item[A3.] Equations~\eqref{betaequations} have a solution $\rvec=(r_1,\ldots,r_n)$
    such that $\dfrac{r_j}{r_k}\le 1+R $ for $jk\in G$,
     where $R=R(n)$ satisfies $0\le R =O(1)$
    and $R^2 \dfrac n\varDelta \lognd =o(\log n)$.
\end{itemize}
Adopt all the definitions in~\eqref{afdefs}.
  Then the number of orientations of $G$ with imbalance sequence $\imbavec$ is
  \begin{equation}\label{answer}
       \pi^{-(n-1)/2} P(G,\imbavec)^{-1}
      \varDelta^{1/2}n^{1/2}\abs{A}^{-1/2} 
          \exp\(\psi(G,\imbavec)+ O(R^3\varDelta^{-3/2+\eps/2}n+\varDelta^{-3+\eps}n)\).
  \end{equation}
\end{theorem}

Note that $R^3\varDelta^{-3/2+\eps/2}n = O(n^{-1/2 + \eps})$
by assumption A3  so the error terms in ~\eqref{answer} are always vanishing.
In the particular case of Eulerian orientations,~$R=0$.  

The  quantities $P(G,\imbavec) $  and $\varDelta^{1/2}n^{1/2}\abs{A}^{-1/2}$
have interesting interpretations. First,
$P(G,\imbavec)$ is {the probability of each orientation with} imbalance sequence $\imbavec$ in the Bradley-Terry model, as we indicate in Lemma~\ref{uniformlemma}.
Second, suppose each edge $jk$ of $G$ is assigned weight $2\lamlam$ and each 
spanning tree of $G$ is assigned weight equal to the product of the weights of its edges. Define $\kappa(G,\rvec)$ to be the sum over all weights of spanning trees in $G$.
Note that the eigenvalues of $A$ are $\varDelta$ (from the term $\dfrac\varDelta nJ$)
together with the non-zero eigenvalues of~$L$.
Therefore, using the Matrix-Tree Theorem (for example, \cite[Theorem 5.2]{Moon1970}), we get
\[
   \varDelta^{1/2}n^{1/2}\abs{A}^{-1/2}=\kappa(G,\rvec)^{-1/2}.
\]

 The quantities  $\E f_4(\X)$,  
 $\E f_6(\X)$, $\Var f_3(\X)$, and $\Var f_4(\X)$ defining  $\psi(G,\imbavec)$ can be calculated by inverting the matrix $A$ and using Isserlis' formula; see Lemma~\ref{l:Isserlis}.  Their growth  rates are given in the next lemma. 
 Note that  if  $\varDelta \geq n^{1/2+\eps}$ and  $r_j/r_k \leq  1 + \varDelta^{1/2}n^{-1/2+\eps}$ for all $j,k$ then   
   $\E f_6(\X)$, $\Var f_3(\X)$, $\Var f_4(\X)$ are vanishing while 
    $\E f_4(\X)$ can be explicitly approximated in terms of the degrees of the graph~$G$.  

  \begin{lemma}\label{l:expvar}
         Let the assumptions A1, A2, A3 of Theorem~\ref{t:bigtheorem} hold. Then,
          \begin{align*}  
                  	 \E f_4(\X) &= -\dfrac14\sum_{jk\in G}\, \(d_j^{-1}+d_k^{-1}\)^2 +
                  	 O\( R^2\varDelta^{-1}n +
          	     \varDelta^{-2}n \lognd\) = O(\varDelta^{-1}n ),
          	   \\
         \Var f_3(\X) &=   O\(R^2 \,\varDelta^{-1}n \lognd\), \qquad
                  	          \E f_6(\X), \Var f_4(\X) = O\(\varDelta^{-2}n \lognd\), 
                           \end{align*}
          where $d_1,\ldots,d_n$ are the degrees of $G$.
    \end{lemma}
For the case when  $\imbavec = \boldsymbol{0}$, we solve  \eqref{betaequations} by setting $r_1= \cdots= r_n$. Thus,  Theorem \ref{t:bigtheorem} and Lemma \ref{l:expvar} immediately  give an asymptotic  formula for the number of Eulerian orientations.  This formula was previously known only for the dense range $\varDelta = \Omega(n)$; see \cite{misha}.

%

\begin{corollary}\label{c:eulerian}
 Let $G=G(n)$ be a graph with even degrees $d_1,\ldots,d_n$, satisfying
 assumptions A1 and~A2 of Theorem~\ref{t:bigtheorem}.
  Then the number of Eulerian orientations of $G$ is
\[
     2^{\card{E(G)} + (n-1)/2} \pi^{-(n-1)/2} \kappa(G)^{-1/2}
       \exp\Bigl(-\dfrac14\sum_{jk\in G}\, \(d_j^{-1}+d_k^{-1}\)^2
                         + O\(\varDelta^{-2}n\lognd \)\Bigr),
\]
where $\kappa(G)$ is the number of (unweighted) spanning trees.
\end{corollary}

We prove Theorem \ref{t:bigtheorem} and Lemma \ref{l:expvar}  in Section \ref{s:proofmain}.    Applications of these results include estimating the probability for a uniform random orientation with  given imbalance sequence to contain a prescribed subdigraph.  For example, one might be interested in estimating the chance that a team $A$ has defeated both teams $B$ and $C$ in a tournament given the scores of all the teams. We give a simple demonstration of such an application in Section \ref{ss:eulerian} (for Eulerian orientations).

 In Section \ref{s:beta-model} we study equations  \eqref{betaequations}.
We provide necessary and sufficient conditions for the existence and the uniqueness (up to scaling) of the solution and find an explicit bound on the ratios $\{ r_j/r_k\}$.
In particular we  obtain a simple  sufficient condition  for
assumption A3 of Theorem~\ref{t:bigtheorem} to hold, stated below.

\begin{theorem}\label{t:sufficient}
  Adopt assumptions A1 and A2 of Theorem~\ref{t:bigtheorem}.   If 
  \[
\imbamax  = o\( \varDelta^{3/2}n^{-1/2}\, \log^{-1}\dfrac{2n}{\varDelta}\),
  \]
then  assumption~A3  of Theorem~\ref{t:bigtheorem} holds with $R = O\Bigl(\dfrac{\imbamax}{\varDelta} \lognd\Bigr)$.
\end{theorem}
Throughout the paper $\norm{\cdot}_p$ stands for the standard vector norm or for the corresponding induced matrix norm.
The proof of  Theorem \ref{t:sufficient} is given at the end of Section \ref{s:beta-model}. 


\nicebreak
\section{The Bradley--Terry model of orientations}\label{s:beta-model}

In this section we explore the existence and nature of solutions to the
balance equations~\eqref{betaequations}.
Except in the proof of Theorem~\ref{t:sufficient}, we do not require
assumptions A1--A3 in this section.
Some of the techniques used in this section follow those of
Barvinok and Hartigan~\cite{BarvHart2}.

Consider a graph $G$ and for each edge $jk\in G$ choose
numbers $p_{jk},p_{kj}$ with $0\le p_{jk},p_{kj}\le 1$ and $p_{jk}+p_{kj}=1$.
Now independently orient each edge $jk$
towards $k$ with probability $p_{jk}$ and towards $j$ with probability~$p_{kj}$.
We call this a \textit{random orientation of $G$ with parameters $\{ p_{jk} \}$}.
It is \textit{degenerate} if some $p_{jk}$ equals~0 or~1.
It is \textit{conditionally uniform} if, for every orientation $\vec G$ of $G$, all the orientations of
$G$ with the same imbalances as $\vec G$ have the same probability.

\begin{lemma}\label{uniformlemma}
   A non-degenerate random orientation of $G$ with parameters $\{ p_{jk} \}$
   is conditionally uniform if and only if there exists $\rvec\in\Reals_+^n$ such that
   $p_{jk}=\ljk(\rvec)$ for all $jk\in G$, where $( \ljk )$ are given by~\eqref{lambdadef}.
\end{lemma}
\begin{proof}
   Let $\imbavec$ be the imbalance sequence of an orientation $\vec G$.
   Then, for a random orientation with parameters $(\ljk )$, $\vec G$
   occurs with probability $P(G,\imbavec)$ {(whether or not~\eqref{betaequations}
   holds)}.
   This proves uniformity.
   
   Conversely, suppose that the non-degenerate
   random orientation with parameters
   $\{ p_{jk} \}$ is conditionally uniform.
   Assume that $G$ is connected (otherwise, apply the following
   argument to each component).
      
   Take a spanning tree $T$, and
   assign a number $r_j$ to each vertex $j$ as follows.  First,
   $r_1:=1$.  Then, for $j \ne 1$, let $1=v_0,v_1,\ldots,v_s=j$ be
   the unique path from $1$ to $j$ in~$T$.
   Define $r_j : = \prod_{t=1}^s \((1-p_{v_{t-1}v_t})/p_{v_{t-1}v_t}\)$.
   Then, using this $\rvec$ to define the parameters $( \ljk )$, we can now check that $p_{jk}=\ljk$ for $jk\in T$.
   Consider an edge $jk\in G\setminus T$ and let
   $u_0,u_1,\ldots,u_s=u_0$ be the unique cycle in~$G$ that
   contains $jk$ and otherwise only edges of~$T$.
   Let $\vec G$ be any orientation of $G$ in which this cycle is
   a directed cycle.  Since reversing the edges on the cycle gives
   the same imbalance sequence as $\vec G$, uniformity implies
   that
   $\prod_{t=1}^s p_{u_{t-1}u_t} = \prod_{t=1}^s (1-p_{u_{t-1}u_t})$. Then, by the definition of $\rvec$, we get that
   $
   		\frac{p_{jk}}{1-p_{jk}} = \frac{r_j}{r_k}.
   $
   This implies that $p_{jk}=\ljk$, and the proof is complete.
\end{proof}

\begin{lemma}\label{l:generalrandom}
A sequence $\imbavec=(\imba_1,\ldots,\imba_n) \in \Reals^n$  is  an expected
imbalance sequence of some random orientation of $G$
if and only if  $\sum_j \imba_j =0$ and 
\begin{equation}\label{capacity}
	\sum_{j\in U} \imba_j \le |\partial_G U|    \qquad \text{for every\/ $U \subseteq V(G)$.} 
\end{equation}
In addition, $\imbavec$ is the expected imbalance sequence of some non-degenerate
random orientation if and only if~\eqref{capacity} holds and is  strict  for
every $U$ that is not a union of connected components of~$G$. 
\end{lemma} 
\begin{proof}
In order to prove the lemma, we consider an equivalent network flow problem, and apply the max-flow min-cut theorem of Ford and Fulkerson~\cite{Ford}.
To this end, given $G$ we define an auxiliary flow network $(F, c, s,t)$ with source $s$ and sink $t$, such that $V(F) = V(G) \cup \{s\} \cup \{t\}$ and $E(F) = E(G) \cup \{(s,v) : v \in V(G)\} \cup \{(t,v), v \in V(G)\}$.
The capacity function $c: V(F) \times V(F) \to \Reals$ is then defined such that, for $u,v \in V(G)$, $c_{sv} := d_{v} + \imba_{v}$, $c_{vt} := d_v$, $c_{uv} = c_{vu} := 1$ and all other capacities are $0$.
Note that every cut in the network has the form $(\{s\}\cup U, \{t\} \cup (V(G) \setminus U))$ for some $U\subseteq V(G)$.
The capacity of this cut is 
\begin{align} \label{eq:capcut}
\sum_{j \in V(G) \setminus U} (d_{j} + \imba_{j}) + |\partial_{G} U| + \sum_{k\in U} d_{k} = 2|E(G)| - \sum_{j \in U} \imba_{j} + |\partial_{G} U|,
\end{align}
where we have used $\sum_{j} d_j = 2|E(G)|$ and $\sum_j \imba_j = 0$. By \eqref{eq:capcut} and the max-flow min-cut theorem (\cite{Ford}, Theorem 1), there is a flow $\fl: V(F) \times V(F) \to \Reals$ of value $2|E(G)|$ iff~\eqref{capacity} holds.
Such a flow saturates all the edges incident to $s$ or $t$, so from each vertex $j\in V(G)$, the net flow on the arcs between $j$ and other vertices in $V(G)$ is $\imba_j$, that is
\begin{align} \label{eq:netflow}
\sum_{k \in N(j)} (\fl(j,k) - \fl(k,j)) = \imba_{j},
\end{align}
where $N(j)$ is the set of neighbours of $j$ in $G$. Now, for $jk \in  G$, define $\{p_{jk}\}$ by 
\[
p_{jk} := \dfrac12(1+\fl(j,k) -\fl(k,j)).
\] Note that $p_{jk} + p_{kj} = 1$  for any $jk\in G$ and, by \eqref{eq:netflow}, the random orientation with parameters $\{p_{jk}\}$ has expected imbalance sequence $\imbavec$.
This proves the first equivalence.

For the second part, suppose that $\imbavec$ is such that \eqref{capacity} holds and is strict for any $U$  such that $\partial_G(U)\ne\emptyset$; that is, it is not a union of connected components of~$G$. Then, there is some $\eps$
with $0<\eps<\frac12$ such that
\[
          \sum_{j\in U} b'_j \le \card{\partial_G(U)}
\]
for all $U\subseteq V(G)$, where $\imbavec' := \dfrac{1}{1-2\eps}\imbavec$.
By the first part of this lemma, there exists a (possibly degenerate) random orientation of $G$ with parameters $\{p'_{jk}\}$ and expected imbalance sequence~$\imbavec'$.
Now define $\{p_{jk}\}$ by $p_{jk} := \eps + (1-2\eps)p'_{jk}$ for $jk\in G$, and note that we
still have $p_{jk} + p_{kj} = 1$ and $\sum_{k \in N(j)} (p_{jk} - p_{kj}) = \sum_{k \in N(j)}(1-2\eps)(p'_{jk} -p'_{kj}) = \imba_{j}$.
That is, $\{p_{jk}\}$ are non-degenerate parameters with expected imbalance sequence~$\imbavec$. 

Conversely, note that any random orientation of $G$ with parameters $\{p_{jk}\}$ induces a maximum flow $\phi$ on the network, by setting $\fl(j,k) = p_{jk}$, and assuming the flow is at maximum capacity on arcs incident to $s$ or $t$.
But, now, if equality occurs in~\eqref{capacity} for some $U$ that $\partial_G(U) \neq \emptyset$, then the cut $(\{s\} \cup U, (V(G)\setminus U) \cup \{t\})$ is saturated by any flow of value $2|E(G)|$, so the edges crossing it must have flow $1$ in one direction and $0$ in the other.
In particular, this implies that the probabilities corresponding to flows on arcs across the cut must be degenerate.
\end{proof}

\begin{theorem}\label{t:existence}
Let $\imbavec=(\imba_1,\ldots,\imba_n) \in \Reals^n$ be such that 
$\sum_j \imba_j =0$ and
\[
	\sum_{j\in U} \imba_j \leq |\partial_G U|   \qquad \text{for every $U \subseteq V(G)$,} 
\]
with the inequality being strict for any $U$ that is not the union of connected components of~$G$.
Then there exists $\rvec=(r_1,\ldots,r_n)\in\Reals^n$, unique up to
uniform scaling in each connected component of $G$,
such that the random orientation of $G$ with parameters
$( \ljk )$ given by~\eqref{lambdadef} has expected imbalance
sequence~$\imbavec$.
\end{theorem}
\begin{proof}
Consider a random orientation of $G$ with parameters $\{p_{jk}\}$.
We view these parameters as a vector $\pvec \in [0,1]^{2|E(G)|}$, and let $S$ be the set of possible directed edges $\vec{jk}$ in an orientation of~$G$.
Then, since the edges of $G$ are oriented independently, the entropy function corresponding to this orientation is given by \[H(\pvec):= -\sum_{\vec{jk} \in S} p_{jk}\log{p_{jk}},\] with the usual convention that the terms corresponding to $0\log{0}$ are $0$.
This  is a continuous function on a compact set, thus there exists a maximiser $\pvec$. 

Next, we show by contradiction  that $\pvec$ is non-degenerate. Assume otherwise.
%
Note that by Lemma \ref{l:generalrandom}, there exists a non-degenerate $\pvec'$  for which 
 the expected imbalance sequence is $\imbavec$.  
Let $A$ be the set of directed edges $\vec{jk}$ such that $p_{jk} = 0$. Then, for $\eps \in (0,1)$,  
\[
  H((1-\eps)\pvec + \eps\pvec') = -\sum_{\vec{jk} \in A}\eps p'_{jk}\log{\eps p'_{jk}} - 
 \sum_{\vec{jk} \in S\setminus A} \((1-\eps)p_{jk} + \eps p'_{jk}\)
 \log{\((1-\eps)p_{jk} + \eps p'_{jk}\)}.
\]
Using the strict concavity of the function $x \mapsto -x\log{x}$ on $[0,1]$, we get 
\begin{align*}
&-\sum_{\vec{jk} \in S\setminus A} \((1-\eps)p_{jk} + \eps p'_{jk}\)
 \log{\((1-\eps)p_{jk} + \eps p'_{jk}\)} \\
& \hspace{4cm} \geq -(1-\eps) \sum_{\vec{jk} \in S \setminus A} p_{jk} \log{p_{jk}} - \eps\sum_{\vec{jk} \in S\setminus A} p'_{jk}\log{p'_{jk}}.
\end{align*}
Using the fact that $H(\pvec) = -\sum_{\vec{jk} \in S \setminus A} p_{jk} \log{p_{jk}}$, this yields the lower bound 
\begin{align*}
H((1-\eps)\pvec + \eps \pvec') \geq H(\pvec) - \eps\biggl(\,\sum_{ \vec{jk} \in A}p'_{jk}\log{\eps p'_{jk}} - \sum_{\vec{jk} \in S\setminus A} (p_{jk}\log{p_{jk}} - p'_{jk} \log{p'_{jk}})\biggr).
\end{align*}
Now, for $\eps$ sufficiently small, the bracketed term on the right can be made negative, which implies $H((1-\eps)\pvec + \eps \pvec') > H(\pvec)$, a contradiction.

Denoting Lagrange multipliers by $\{\beta_j\}$, define 
\[
\tilde{H}(\pvec) = H(\pvec) + \sum_{j=1}^{n} \beta_j\biggl(\,\sum_{k \in N(j)} (p_{jk} - p_{kj})  - \imba_j\biggr),
\]
and consider this is a function of $\card{E(G)}$ variables $p_{jk}$ for $jk\in G$,
where one of $p_{jk}$ and $p_{kj}$ is arbitrarily chosen and the other is
determined by $p_{jk}+p_{kj}=1$.
The partial derivatives satisfy 
\begin{equation}\label{derivs}
   \frac{\partial \tilde{H}(\pvec)}{\partial p_{jk}} = -\log{\frac{p_{jk}}{1 - p_{jk}}} + \beta_j - \beta_k.
\end{equation}
By setting these partial derivatives to $0$, we find that the maximiser $\pvec$ satisfies \[p_{jk} = \frac{e^{\beta_j}}{e^{\beta_j} + e^{\beta_k}},\] so that if we set $r_j = e^{\beta_j}$ 
for $1\le j\le n$, the corresponding random orientation has parameters $(\ljk)$ as defined by~\eqref{lambdadef}.
Moreover, by the strict concavity of the entropy function, on the convex, compact set corresponding to the equality constraints, the maximiser $\pvec$ is unique.
This implies by~\eqref{derivs} that for $jk \in G$ the ratios $r_j/r_k$ are unique,
so that the $r_{j}$ are unique up to uniform scaling in every connected component of $G$.
\end{proof}

\begin{lemma}\label{l:tameness}
 Let $G$ be a connected graph
  of maximum degree $\varDelta$. 
 Let $\imbavec \in \Reals^n$ and $0 < \delta \leq 1$ be  such that $\sum_j \imba_j =0$ and 
 \[	
	 \Bigl|\, \sum_{j\in U} \imba_j \,\Bigr| 
	 \le (1-\delta)\, \abs{\partial_G U} \qquad \text{for any $U \subset V(G).$}
 \]
  Then, for $n \ge 10$, 
  the solution  $\rvec$  of  the system \eqref{betaequations}  is  such that, for all $j$ and $k$, 
  \[
  	\Abs{\log \dfrac{r_j}{r_k} } \le \dfrac{35\,\varDelta}{\delta h(G)} \log \dfrac{n}{\delta h(G)} \log \dfrac{1}{\delta}\,. 
  \]
\end{lemma}
We defer the proof of Lemma~\ref{l:tameness} until Section~\ref{ss:A1}.

\begin{proof}[Proof of Theorem~\ref{t:sufficient}]
  Since $\sum_{j=0}^n \imba_j=0$, we have for any $U\subseteq V(G)$
  that
  \[
      \Bigl| \sum_{j\in U} \imba_j \Bigr| \le \imbamax \min\{\card{U},n-\card{U}\}
      \le \frac{\imbamax}{h(G)}\, \card{\partial_G U}.
  \]
By assumptions, we can bound
   \[
     \frac{\imbamax}{h(G)} 
      = o\biggl( \frac{ \varDelta^{3/2} n^{-1/2} \log^{-1} \dfrac{2n}{\varDelta}}{\gamma \varDelta}\biggr) = o\(\log^{-1} \dfrac{2n}{\varDelta}\).
   \]
Applying Lemma~\ref{l:tameness} with $\delta = 1 -  \dfrac{\imbamax}{h(G)}$, we find that 
\[
   \Abs{\log \dfrac{r_j}{r_k} } = O\biggl(\lognd 
      \log \Bigl(1-\dfrac{\imbamax}{h(G)}\Bigr)^{\!-1}\biggr) = o(1).
 \]
  Thus, we get that
  $\dfrac{r_j}{r_k}  =1+o(1) $ and  so
  \[
       R = o(1) \qquad \text{and} \qquad R^2 \dfrac{n}{\varDelta} \lognd = o(\log  n).
   \]
  This completes the proof of that  assumption~3 holds.
\end{proof}

\nicebreak
\section{Enumeration}
\label{s:enumeration}
The \textit{Laplacian matrix} of $G$ is the symmetric matrix given by the diagonal matrix of degrees minus the adjacency matrix of~$G$.
Since the row sums of this matrix are zero, it has a zero eigenvalue
corresponding to an eigenvector with all components equal.
The next smallest eigenvalue, $\lambda_2(G)$, is called the
\textit{algebraic connectivity} of~$G$ and is closely related
to the Cheeger constant.

\begin{lemma}[{\cite{Mohar1989}}]\label{l:isoperimetric}
For any graph $G$, we have
\[
	\dfrac12 \lambda_2(G)\le h(G) \le \sqrt{\lambda_2(G) (2 \varDelta - \lambda_2(G))}.
\]
\end{lemma}

\begin{lemma}\label{l:consequences}
  Under assumptions A1--A3, the following are true.
  \begin{itemize}\itemsep=0pt
  \item [(a)]The minimum degree of $G$ is at least $\gamma\varDelta$.
  \item[(b)] $\lambda_2(G)\ge \(1-(1-\gamma^2)^{1/2}\)\varDelta
     \ge\frac12\gamma^2\varDelta$.
  \item[(c)] For $jk\in G$,
  $\dfrac{1+R }{(2+R)^2} \le \lamlam \le \dfrac14$
      \,and \,$\abs{\lamdiff}\le \dfrac{R }{2+R }=O(R)$.
   \end{itemize}
\end{lemma}
\begin{proof}
Part (a) follows from the trivial fact that $h(G)$ cannot be larger
than the minimum degree.
Part (b) follows from Lemma~\ref{l:isoperimetric}.
Part (c) is a simple consequence of~A3.
\end{proof}

Let $N(G,\imbavec)$ be the number of orientations of $G$ with
imbalance sequence~$\imbavec$.
By Cauchy's integral formula, using the generating function
$\prod_{jk\in G} \Bigl( \dfrac{x_j}{x_k} + \dfrac{x_k}{x_j} \Bigr)$, we have
\begin{align*}
   N(G,\imbavec) &= [x_1^{\imba_1}\cdots x_n^{\imba_n}]
       \prod_{jk\in G} \Bigl( \dfrac{x_j}{x_k} + \dfrac{x_k}{x_j} \Bigr) \\
       &= \frac{1}{(2\pi i)^n} \oint\cdots\oint 
         \frac{\prod_{jk\in G} \( x_j/x_k + x_k/x_j \)}
                {x_1^{\imba_1+1}\cdots x_n^{\imba_n+1}}
                \,dx_1\cdots dx_n,
\end{align*}
where the contours circle the origin once anticlockwise.
We choose the circles $x_j=r_j^{1/2} e^{i\theta_j}$ as contours, so that
\begin{align}
   N(G,\imbavec) &= (2\pi)^{-n} P(G,\imbavec)^{-1}
             \int_{-\pi}^{\pi}\!\!\cdots\int_{-\pi}^{\pi} F(\thetavec)\,
             d\thetavec, 
\intertext{where $P(G,\imbavec)$ is defined in~\eqref{afdefs},}
  F(\thetavec) &:= e^{-i\sum_{j=1}^n \imba_j\theta_j}\,  
      \prod_{jk\in G} f_{jk}(\theta_j-\theta_k), 
      \notag \\[-0.8ex]
   f_{jk}(x) &:= \frac{e^{ix}}{1+r_k/r_j}+\frac{e^{-ix}}{1+r_j/r_k}.  \label{fdef} 
\end{align}

Given $x\in\Reals$, define
\[
    \semiabs{x} := \min\{\abs{x-k\pi} \st k\in \Integers \}.
\]
It is easily seen that $\semiabs{\,\cdot\,}$ is a seminorm on $\Reals$ that induces a
norm on $\Rmodpi$, the real numbers modulo~$\pi$.
An \textit{interval} of $\Rmodpi$ of \textit{length} $\rho\ge 0$ is a set of the form
\[
    I(x,\rho) := \{ \xi \in \Rmodpi \st \semiabs{x-\xi} \le \nfrac 12\rho \}.
\]
We will also write $I(x,\rho)$ as $[x-\frac12\rho,x+\frac12\rho]$ when it
is not ambiguous. 

Next, note that any individual value $\theta_j$ can be replaced by
$\theta_j+\pi$ without
changing $F(\thetavec)$, since in every orientation the imbalance
of a vertex has the same parity as its degree in~$G$.
This means we can write
\begin{equation}\label{NGint}
   N(G,\imbavec) = \pi^{-n} P(G,\imbavec)^{-1}\, J', \text{~~~where~~~}
             J' := \int_{(\Rmodpi)^n} F(\thetavec)\,
             d\thetavec.
\end{equation}
We will approach~\eqref{NGint} by splitting the region of integration $(\Rmodpi)^n$
in several parts.
Let
\begin{align*}
    \varOmega_0 &:= \bigl\{\thetavec \in (\Rmodpi)^n \st 
      \text{there exists $x\in\Rmodpi$ such that $\thetavec\in I(x,\varDelta^{-1/2}\log^4 n)^n$}\bigr\} 
\\
     J_0 &:=   \int_{  \varOmega_0} F(\thetavec)\,
             d\thetavec.
\end{align*}
In other words, the region $\varOmega_0$ consists of those $\thetavec \in (\Rmodpi)^n$
such that all components $\theta_j$ can be covered by an interval of $\Rmodpi$ of length
at most  $\varDelta^{-1/2} \log^4 n$.
It will turn out that $J_0$ will dominate $J'$, and that in the complement
of $\varOmega_0$ even the integral of $\abs{F(\thetavec)}$ is negligible.

\nicebreak
\subsection{The integral inside $\varOmega_0$}\label{s:inbox}

We are going to apply the techniques developed in~\cite{Mother}.
For any $c$, define $U_n(c) = I(0,c)^n$.
The assumptions of Theorem~\ref{t:bigtheorem} hold throughout
this section.

First note that, since $\sum_j\imba_j=0$, we can uniformly translate
each $\theta_j$ without changing $F(\thetavec)$.
Also,
\begin{align*}
    \{ \thetavec\in (\Rmodpi)^n\st \semiabs{\theta_j-{}&\theta_n}
        \le \varDelta^{-1/2}\log^4 n \text{, $1\le j\le n{-}1$} \} \\
    &\subseteq \varOmega_0\subseteq
     \{ \thetavec\in (\Rmodpi)^n \st \semiabs{\theta_j-\theta_n}
        \le 2 \varDelta^{-1/2}\log^4 n \text{, $1\le j\le n{-}1$}\}.
\end{align*}
Therefore, if we define $\thetapvec=(\theta_1,\ldots,\theta_{n-1},0)$, we have
an $(n{-}1)$-dimensional integral:
\begin{equation}\label{NGint2}
   J_0=\pi \int_{\varOmega'} F(\thetapvec)\, d\thetapvec,
\end{equation}
for some region $\varOmega'$ with $U_{n-1}(\varDelta^{-1/2}\log^4 n)
\subseteq\varOmega'\subseteq U_{n-1}(2\varDelta^{-1/2}\log^4 n)$.

Next we lift the integral back to full dimension using~\cite[Lemma~4.6]{Mother},
which we quote for convenience as Lemma~\ref{LemmaQW}.
Let $M$ be the matrix with 1 in the last column and 0 elsewhere.
Define:
\begin{align*}
   \rho_1&=\varDelta^{-1/2}\log^4 n,~~\rho_2=2\varDelta^{-1/2}\log^4 n,~~
   \rho=\log^4 n \\
   P &= I - \dfrac1n J,~~
   Q = I - M,~~
   S = \varDelta^{-1/2} I \text{ and }
   W = \varDelta^{1/2}n^{-1} J.
\end{align*}
One can easily check that $PQ+SW=I$, and also that
$\ker Q\cap\ker W=\{\boldsymbol0\}$, $\ker Q$ has dimension~1 and
$\sp(\ker Q,\ker W)=\Reals^n$.
We also have $\abs{Q\trans\!Q+W\trans W}=n\varDelta$,
$\kappa=1$, $\infnorm{P}\le 2$, $\infnorm{Q} = 2$,
$\infnorm{S}=\varDelta^{-1/2}$ and $\infnorm{W}=\varDelta^{1/2}$.
Now applying \cite[Lemma~4.6]{Mother},
and the fact that $F(\thetavec)$ is invariant under translating each coordinate,  we have
\begin{align*}
   J_0 &= \(1+O(n^{1-\log^7 n})\)\, \pi^{1/2} (\varDelta n)^{1/2}
            \int_{\varOmega}
             \hat F(\thetavec)\,
             d\thetavec, 
   \end{align*}          
    where $\varOmega$ is  a region such that  $U_n(\tfrac12\, \varDelta^{-1/2}\log^4 n) \subseteq \varOmega
          \subseteq U_n(5 \varDelta^{-1/2}\log^4 n)$
          and
    \begin{align*} 
        \hat F(\thetavec) := e^{-\frac \varDelta n(\theta_1+\cdots+\theta_n)^2} F(\thetavec).
\end{align*}

\begin{lemma}\label{l:fexpansion}
For $\thetavec\in\varOmega$, we have
\begin{align}
  \log \hat F(\thetavec) &= - \thetavec\trans\! A\thetavec 
    + i \(f_3(\thetavec)+f_5(\thetavec)\) + f_4(\thetavec) +f_6(\thetavec) + \rem(\thetavec),\notag\\
  \intertext{where $A$, $f_3, f_4$ and $f_{6}$ are as defined in~\eqref{afdefs},}
     f_5(\thetavec) &:= - \dfrac4{15} \sum_{jk\in G} \lamlam(\lamdiff)(1-12\lamlam)
          (\theta_j-\theta_k)^5 \text{ and} \notag\\
     \rem(\thetavec) &:=
         O(R\varDelta^{-5/2}n\log^{28} n+\varDelta^{-3}n\log^{32} n).\label{remainder}
\end{align}
\end{lemma}
\begin{proof}
Note that the definitions of $f_{jk}$ in \eqref{fdef} and $\lambda_{jk}$ in \eqref{lambdadef} imply that
\[
	f_{jk}(x) - 1 = \lambda_{jk} (e^{ix}-1) + \lambda_{kj} (e^{-ix}-1).
\]
By Taylor's Theorem and Lemma~\ref{l:consequences}, 
for $\abs{x} \le \varDelta^{-1/2}\log^4 n$, we have
\begin{align*}
   \log f_{jk}(x) &= i(\lamdiff)x - 2\lamlam x^2
     + \dfrac43i (\lamdiff)\lamlam x^3 \\
     &{\quad}+ \dfrac23 \lamlam(1-6\lamlam) x^4 
   -\dfrac4{15}i(\lamdiff)\lamlam(1-12\lamlam) x^5 \\
     &{\quad}- \dfrac4{45}\lamlam(1-30\lamlam+120\ljk^2\lkj^2) x^6 \\
     &{\quad}+ O(R\varDelta^{-7/2}\log^{28} n + \varDelta^{-4}\log^{32}n \).
\end{align*}
Summing $\log f_{jk}(\theta_j-\theta_k)$ over $jk\in G$, and subtracting
$i\sum_{j=1}^n \imba_j\theta_j$, we find that the linear term cancels because
of~\eqref{betaequations} and the error term is as stated because of
Lemma~\ref{l:consequences}(c).
\end{proof}

\nicebreak
\begin{lemma}\label{l:norms}
Consider the symmetric positive-definite matrix $A$ defined in \eqref{afdefs}.
Then the following are true.
\begin{itemize}\itemsep=0pt
  \item[(a)] $\infnorm{A^{-1}}=O\(\varDelta^{-1}\lognd \)$.
  \item[(b)] If $A^{-1}=(a_{jk})$, then $a_{jj}=O(\varDelta^{-1})$ and
    $a_{jk}=O\(\varDelta^{-2}\lognd \)$ uniformly for\\ $1\le j\ne k\le n$.
  \item[(c)] There exists a symmetric positive-definite matrix $T$ such that
    $T\trans\!AT=I$.  Moreover, $\infnorm{T}=O(\varDelta^{-1/2}\log^{1/2}\!n)$
    and $\infnorm{T^{-1}}=O(\varDelta^{1/2})$.
\end{itemize}
\end{lemma}
\begin{proof}
  Part (a) follows from assumption A2 and Lemmas~\ref{l:consequences} and 
  \ref{l:inversenorm}.
  To prove Part(b), let $D$ be the diagonal of~$A$.
  We have $A^{-1}-D^{-1}=A^{-1}(D-A)D^{-1}$, so the maximum absolute value of
  an entry of $A^{-1}-D^{-1}$ is bounded by $\infnorm{A^{-1}}$
  times the maximum absolute value of an entry of $(D-A)D^{-1}$.
  The claim thus follows from Part~(a).
  Both bounds in Part~(c) come from
  Corollary~\ref{c:Apowers} when we take $T=A^{-1/2}$ and note that
  $\bigl| \binom{-1/2}k\bigr|<k^{-1/2}$ and $\bigl| \binom{1/2}k\bigr|<k^{-3/2}$
  for $k\ge 1$.
\end{proof}

We will also use the following simple applications of Isserlis' formula~\cite{Isserlis}.
\begin{lemma}\label{l:Isserlis}
  Let $Z$ and $(Z_1,Z_2)$ be normal random variables with zero mean.
  For integer $m$, let $p(m)$ be the number of ways to divide $m$
  things into $m/2$ pairs (i.e., 0 for odd~$m$ and $(m-1)!!$ for
  even~$m$). Then, for integers $s,t\ge 0$,
  \begin{itemize}\itemsep=0pt
     \item[(a)] $ \E Z^s = p(s)(\Var Z)^{s/2}$.
     \item[(b)] $\displaystyle
        \Cov(Z_1^s,Z_2^t) \\[0.5ex]
        = \sum_{u=1}^{\min\{s,t\}}
         \binom su\binom tu u!\, p(s-u)p(t-u)(\Var Z_1)^{(s-u)/2}
            (\Var Z_2)^{(t-u)/2}\Cov(Z_1,Z_2)^u $.\qed
  \end{itemize}
\end{lemma}

Let $\X=(X_1,\ldots,X_n)$ be a random vector with normal density
$\pi^{-n/2} \abs{A}^{1/2} e^{-\xvec\trans\! A\xvec}$.
The covariance matrix of $\X$ is $(\sigma_{jk})= (2A)^{-1}$.
For $jk\in G$, define $Y_{jk}:=X_j-X_k$.
Then the vector $\Y:=(Y_{jk})_{jk\in G}$ also has a normal density
with zero mean;  let $\varSigma=(\varsigma_{jk,j'k'})$ denote 
its covariance matrix.
\begin{lemma}\label{l:Ycovmatrix}
We have the following. 
\begin{itemize}\itemsep=0pt
\item[(a)] For $jk,j'k'\in G$,
 \[
   \varsigma_{jk,j'k'} = \sigma_{jj'}+\sigma_{kk'}-\sigma_{jk'}-\sigma_{kk'}
      = \begin{cases}
          O\(\varDelta^{-2}\lognd \),
                          & \text{ if $\{j,k\}\cap\{j',k'\}=\emptyset$};\\[0.5ex]
          O(\varDelta^{-1}), 
               & \text{ if $\{j,k\}\cap\{j',k'\}\ne\emptyset$}.
        \end{cases}
 \]
 \item[(b)] ~$\infnorm{\varSigma} = O\(\lognd \)$.
 \item[(c)] For integers $\ell\ge 1$ and $jk\in G$,
   \[  \E Y_{jk}^\ell = \begin{cases}
      0, & \text{ if $\ell$ is odd}; \\
      O(\varDelta^{-\ell/2}), & \text{ if $\ell$ is even}.
          \end{cases}
   \]
  \item[(d)] For integers $\ell,\ell'\ge 0$ and $jk\in G$,
   \[
       \sum_{j'k'\in G} \Cov(Y_{jk}^\ell,Y_{j'k'}^{\ell'})
        = \begin{cases}
           0, & \text{ if $\ell+\ell'$ is odd}; \\
           O\(\varDelta^{1-(\ell+\ell')/2} \lognd \),
              & \text{ if $\ell+\ell'$ is even}.
          \end{cases}
   \]
\end{itemize}
\end{lemma}
\begin{proof}
  Part (a) follows from Lemma~\ref{l:norms}(b).
  For~(b), note that $\sum_{j'=1}^n \sigma_{jj'}\le \infnorm{(2A)^{-1}}$
  and that there at most $\varDelta$ choices of~$k'$ for each~$j'$.
  The other terms are similar, so the result follows on applying Lemma~\ref{l:norms}(a).
  
  Part~(c) follows from Part~(a) and Lemma~\ref{l:Isserlis}(a).  We use
  Lemma~\ref{l:Isserlis}(b) for Part~(d): bound all variances and
  covariances except $\Cov(Y_{jk},Y_{j'k'})$ by $O(\varDelta^{-1})$
  (on account of Part~(a)) and then using Part~(b) to bound the sum of these terms
  over $j'k'\in G$.
\end{proof}

Define $\fre(\xvec):=f_4(\xvec)+f_6(\xvec)$,
$\fim(\xvec):=f_3(\xvec)+f_5(\xvec)$, and
$f(\xvec):=i \fim(\xvec) + \fre(\xvec)$.

\begin{lemma}\label{t:J0value}
We have
\begin{align*}
    J_0 &= \pi^{(n+1)/2} \varDelta^{1/2}
               n^{1/2} \abs{A}^{-1/2} \\
              &\times \exp \Bigl(\E \fre(\X)-\dfrac12 \Var f_3(\X) + \dfrac12\Var f_4(\X)
               + O(R^3\varDelta^{-3/2+\eps/2}n+\varDelta^{-3+\eps}n)\Bigr).
\end{align*}
\end{lemma}
\begin{proof}
We will apply \cite[Theorem 4.4]{Mother} which, for
convenience, we quote in Section \ref{A:integ} as Theorem~\ref{t:gauss4pt}.

By Lemma~\ref{l:norms}(c), there are constants $c_1,c_2>0$ such that
$U_n(\hat\rho_1) \subseteq T^{-1}\varOmega\subseteq U_n(\hat\rho_2)$,
where $\hat\rho_1:=c_1\log^{7/2}n$ and $\hat\rho_2:=c_2\log^4n$.

Next, note that
$\xvec\in\varOmega\implies\infnorm{\xvec}=O(\varDelta^{-1/2}\log^4 n)$.
Under this condition we calculate that, uniformly over $j,k$,
\begin{align*}
   \Bigl|\dfrac{\partial f(\xvec)}{\partial x_j}\Bigr| &= O(R\log^8 n + \varDelta^{-1/2}\log^{12} n), \\
   \Bigl|\dfrac{\partial^2 f(\xvec)}{\partial x_j\partial x_k}\Bigr|
                                   &= \begin{cases}
                                         O(R\varDelta^{1/2}\log^4n + \log^8 n),
                                            & \text{~~if $j=k$}; \\
                                         O(R\varDelta^{-1/2}\log^4n + \varDelta^{-1}\log^8 n),
                                            & \text{~~if $jk\in G$}; \\
                                         0,    & \text{~~otherwise}.
                                      \end{cases}
\end{align*}
and conclude that Theorem~\ref{t:gauss4pt}(b) holds for
$\phi_1 = R\varDelta^{-1/2+\eps/12}n^{1/3}+\varDelta^{-1+\eps/4}n^{1/3}$ (note that here we incorporate powers of $\log{n}$ into the $\varDelta^{\eps}$ terms).

Now take $g(\xvec):=\fre(\xvec)$.
For Theorem~\ref{t:gauss4pt}(c) we have
$\infnorm{\xvec}=O(\varDelta^{-1/2}\log^{9/2}n)$.
The required derivative bounds are
\begin{align*}
       \Bigl|\dfrac{\partial g(\xvec)}{\partial x_j}\Bigr| &= O(\varDelta^{-1/2}\log^{27/2} n), \\
       \Bigl|\dfrac{\partial^2 g(\xvec)}{\partial x_j\partial x_k}\Bigr| &= \begin{cases}
                                         O( \log^9 n),
                                            & \text{~~if $j=k$}; \\
                                         O(\varDelta^{-1}\log^9 n),
                                            & \text{~~if $jk\in G$}; \\
                                         0,    & \text{~~otherwise},
                                      \end{cases}
\end{align*}
so Theorem~\ref{t:gauss4pt}(c)(ii) is satisfied by $\phi_2=\varDelta^{-1+\eps/4}n^{1/3}$.

The appearance $e^{\Var \fim(\X)}$ in the error term of Theorem~\ref{t:gauss4pt}
is the main reason~$R $ cannot easily be made larger.
Since the coefficients of $f_3(\X)$ and $f_5(\X)$ are $O(R)$,
we have $\Var\fim(\X)=O\(R^2\varDelta^{-1}n\lognd \)
= o(\log n)$ by Lemma~\ref{l:Ycovmatrix}(d) and assumption~A3.
Therefore, $e^{\Var \fim(\X)}=n^{o(1)}=o(\varDelta^{\eps/4})$.

The bound $\rem(\X)=O\(R\varDelta^{-5/2+19\eps/24}n+\varDelta^{-3+\eps/2}n\)$
follows from~\eqref{remainder}.
Putting everything together, the error term $K$ given by Theorem~\ref{t:gauss4pt}
has magnitude
\begin{equation}\label{Kvalue}
   O\(R^3\varDelta^{-3/2+\eps/2}n + R\varDelta^{-5/2+5\eps/6} n + \varDelta^{-3+\eps}n\).
\end{equation}

We can now see that some contributions to $\E f(\X)$ and $\E\, (f(\X) - \E f(\X))^2$ are negligible.
By Lemma~\ref{l:Ycovmatrix},
$\Cov(f_3(\X),f_5(\X))=O\(R^2\varDelta^{-2}n\lognd \)$,
which is less than the geometric mean of the first two terms of~\eqref{Kvalue}
and so is bounded by the larger of them.
Similarly, $\Cov(f_4(\X),f_6(\X))=O\(\varDelta^{-3}n\lognd \)$, and can thus be incorporated into the third term of~\eqref{Kvalue}. The contributions of
$\Var f_5(\X)$ and $\Var f_6(\X)$ are even smaller.

Next, we can remove the middle term of~\eqref{Kvalue} since
$(R^3\varDelta^{-3/2+\eps/2}n)^{1/3}(\varDelta^{-3+\eps}n)^{2/3}
=R\varDelta^{-5/2+5\eps/6}n$. Finally, assumption~A3 implies
that $R^3\varDelta^{-3/2+\eps/2}n=O(n^{-1/2+\eps})$.
This completes the evaluation of the integral~$J_0$.
\end{proof}

We will also need the following bound.
\begin{lemma}\label{J0rough}
We have
\[
     \int_{\varOmega_0} \abs{F(\thetavec)}\,d\thetavec = e^{o(\log n)} J_0
       = e^{O(n\log n)}. \qquad\qed
\]
\end{lemma}
 \begin{proof}
 	Revisiting the proof of Lemma~\ref{t:J0value}, note that the difference
between the integrals of $F(\thetavec)$ and $\abs{F(\thetavec)}$ came
only from $\fim(\xvec)$ and amounted to a factor of $e^{o(\log n)}$.
This implies the first equality.

Observe that all of the eigenvalues of $A^{-1}$ are bounded below by $\infnorm{A}^{-1}$
 and bounded above by $\infnorm{A^{-1}}$. Using  Lemma \ref{l:norms}(a),
 we find that
 $
  |A|^{-1/2} = e^{O(n \log n)}
 $. 
 The remaining factors in the expression for $J_0$ in Lemma \ref{t:J0value} are also $e^{O(n \log n)}$. 
 The bounds 
  \[ \E \fre(\X),  \Var f_3(\X), \Var f_4(\X) = O(n \log n)
  \]
  follow by assumption A3, applying  Lemma~\ref{l:Ycovmatrix}.
  Thus, we get the second equality  from the first.
 \end{proof}

\nicebreak
\subsection{The integral outside $\varOmega_0$}

The conditions of Theorem~\ref{t:bigtheorem} are assumed throughout
this section. We begin with a few lemmas.
\begin{lemma}\label{l:fapp}
For $jk\in G$, $\abs{f_{jk}(x)}$ is a decreasing function of $\semiabs{x}$
  with $f_{jk}(0)=1$ and
\begin{align}
  \abs{f_{jk}(x)}^2
   &= 1 - 4\lamlam \sin^2 x 
                   \leq e^{-\Omega(\semiabs{x}^2)}. \label{fapprox}
\end{align}
In addition, for any $\semiabs{y} \le \semiabs{x}$, we have 
\begin{equation}\label{fapprox2}
	\abs{f_{jk}(x)}
	\le \abs{f_{jk}(y)}\, e^{-\Omega(  (\semiabs{x}^2 - \semiabs{y}^2) (\pi - \semiabs{x} - \semiabs{y}))}.
\end{equation}
\end{lemma}
\begin{proof}
The first part of~\eqref{fapprox} follows from the definition of $f_{jk}(x)$
and implies that
$\abs{f_{jk}(x)}=f_{jk}(\semiabs{x})$ for all $x$.
Therefore we can assume that $0\le y\le x\le \frac12\pi$,
which implies that $\semiabs{x}=x$ and $\semiabs{y}=y$.
Also, recall from Lemma~\ref{l:consequences}(c) that
$\cmin\le 4\lamlam\le 1$ for some constant $\cmin>0$.
Note that, by the concavity of $\cos x$ on $[0,\frac{\pi}{2}]$, we have
$\cos x \geq 1 - \frac{2x}{\pi}$ on this range, which in turn implies
(by symmetry about the line $x = \frac{\pi}{2}$) that 
\begin{align} \label{sinlower}
    \sin{x} \geq \dfrac1\pi x(\pi-x), \quad x \in [0, \pi ].
\end{align}
This in turn implies that $\sin^2 x=\Omega(x^2)$
for $x\in[0,\frac12\pi]$, and combining this with the inequality
$\log z \le z-1$ for all $z > 0$, we have 
$\abs{f_{jk}(x)}^2 \le \exp\( -\Omega(x^2)\)$.

Inequality~\eqref{fapprox2} is trivial if $x=y$, so assume that
$0\le y < x\le\frac12\pi$. In that case, $f_{jk}(y)\ne 0$ and, since
$(1-c\sin^2 x)/(1-c\sin^2 y)$ is a decreasing function of~$c$
for fixed~$x,y$ on this range
\[
      \frac{\abs{f_{jk}(x)}}{\abs{f_{jk}(y)}}
      \le  \frac{1-\cmin \sin^2 x}{1-\cmin\sin^2 y}
      \le \exp \biggl( -\frac{\cmin(\sin^2 x-\sin^2 y)}{1-\cmin\sin^2 y} \biggr)
      \le \exp \( -\cmin(\sin^2 x-\sin^2 y) \).
\]
Finally, by \eqref{sinlower}, we have
\[ 
  \sin^2 x - \sin^2 y = \sin{(x+y)}\sin{(x-y)} \geq \dfrac{1}{\pi^2} (x^2 - y^2) (\pi - x + y) (\pi - x - y)
\]
for $0\le y\le x\le \frac12\pi$, which completes the proof of~\eqref{fapprox2}.
\end{proof}

\begin{lemma}\label{hitamount}
  Let $U,U'$ be disjoint subsets of $\{1,\ldots,n\}$. Suppose
  $\thetavec\in [-\pi,\pi]^n$ such that $\semiabs{\theta_j-\theta_k}\ge x$
  whenever $j\in U,k\in U'$,  for some $x=o(1)$.  Then
\[|F(\thetavec)|
      \le
      \exp\(-\Omega(\varDelta x^2 \log^{-2}n\min\{ \abs U, \abs{U'} \})\).
  \]
\end{lemma}
\begin{proof}
  Consider any of  the paths $v_0,v_1,\ldots,v_\ell$ provided by 
  Lemma~\ref{l:paths}.  By assumption,
  $\semiabs{\theta_{v_0}-\theta_{v_\ell}}\ge x$.
   Since $\ell=O(\log n)$ and
  $\semiabs{\,\cdot\,}$ is a seminorm, we find that 
  \[
  	\sum_{j=1}^\ell \,\semiabs{\theta_{v_j}- \theta_{v_{j-1}}}^2 
  	\ge   \dfrac1\ell \biggl(\sum_{j=1}^\ell \semiabs{\theta_{v_j}- \theta_{v_{j-1}}}\biggr)^{\!2}  
	= \Omega (x^2 \log^{-1} n).
  \]
  Multiplying the bound~\eqref{fapprox} over all the edges of all the paths
  given by Lemma~\ref{l:paths} completes the proof.
\end{proof}

\noindent Define
\[
	 \srho := \varDelta^{-1/2} \log^2 n, \text{~~~and~~~} \brho := \varDelta^{-1/2} \log^4 n.
\]
First, we bound the integral of $\abs{F(\thetavec)}$  in the region
\[
	\varOmega_1 :=  \bigl\{ \thetavec \in (\Rmodpi)^n \st \text{for every
	   $\xi\in\Rmodpi$ we have  $\card{\{ j \st \theta_j\in I(\xi,\srho) \}} <\tfrac45n$} \bigr\}. 
\]

\begin{lemma}\label{separator}
 Suppose $0<t<\frac 13\pi$ and $q\le \frac 15n$.
 Let $X=\{x_1,\ldots,x_n\}$ be a multisubset of $\Rmodpi$ such that
 no interval of
 length $3t$ contains $n-q$ or more elements of~$X$.
 Then there is some interval $I(x,\rho)$, $\rho<\frac13\pi$, such that both 
 $I(x,\rho)$ and $\Rmodpi-I(x,\rho+t)$ contain at least $q$ elements
 of~$X$.
\end{lemma}
\begin{proof}
 Since the conditions and conclusion are invariant under translation,
 we can assume without loss of generality that $[t,2t]$
 is an interval with the greatest number of elements of~$X$ out of all
 intervals of length~$t$.  Since $\Rmodpi-[0,3t]$ has at
 least $q$ elements of $X$ by assumption, $[t,2t]$
 satisfies the requirements of the lemma unless it 
 contains less than~$q$ elements of~$X$.
 
Therefore, assume that all intervals of length~$t$ have less than~$q$
 elements of~$X$.
 For $0\le y \le \pi-3t$, let $\phi(y)$ be the number of elements of~$X$
 that lie in $[t,2t+y]$.  Note that $\phi(y)$ is a non-decreasing step function
 with steps
 of size less than~$q$, also that $\phi(0)<q$ and $\phi(\pi-3t)>n-2q$.
 Therefore, there is some~$y$ such that $\frac12n-\frac32q\le \phi(y)
 \le\frac12n-\frac12q$.
 It can now be checked that $[t,2t+y]$ satisfies the lemma.
\end{proof}

\begin{lemma}\label{S1bound}
We have
\[
    \int_{\varOmega_1} \abs{F(\thetavec)}\,d\thetavec = e^{-\Omega(n\log^2 n)} J_0.
\]
\end{lemma}
\begin{proof}
   If $\thetavec\in \varOmega_1$, the definition of $\varOmega_1 $ implies
   that every interval of $\Rmodpi$ of
  length $\srho$ has fewer than $\frac45n$ components of~$\thetavec$.
  Applying Lemma~\ref{separator} with $t=\frac13 \srho$,
  $q=\frac15n$, and $X = \thetavec$ tells us that there exist $p\in \Rmodpi$ and $s<\frac\pi 3$
  such that both  $I(p,s)$ and $\Rmodpi-I(p,s+t)$ contain at least $\frac15n$
  components  of $\thetavec$.
  For such~$\thetavec$, Lemma~\ref{hitamount}, with $x=t$ and $U, U'$ corresponding
  to the indices of the elements of $\thetavec$ belonging to $I(p,s)$ and $\Rmodpi-I(p,s+t)$ respectively, tells us that
  $\abs{F(\thetavec)} \le \exp\(-\Omega(1)\varDelta t^2n\log^{-2}n\)
  =e^{-\Omega(n\log^2 n)}$.
  Using $\pi^n$ as a bound on the
  volume of $\varOmega_1$, the result follows from Lemma~\ref{J0rough}.
\end{proof}

Next, we bound the integral of $\abs{F(\thetavec)}$  in the region
\[
   \varOmega_2 :=  \bigl\{ \thetavec \in (\Rmodpi)^n \st \text{for some~}
         x\in\Rmodpi \text{ we have }
       \card{\{ j \st \theta_j\in I(x,e^{-\log^3 n}) \}} \ge\tfrac45n \bigr\}. 
\]
\begin{lemma}\label{l:Omega2bound}
We have
\[ 
  \int_{\varOmega_2} \abs{F(\thetavec)}\,d\thetavec = e^{-\Omega(n\log^3 n)}\, J_0.
\]
\end{lemma}
\begin{proof}
  The volume of $\varOmega_2$ is only $e^{-\Omega(n\log^3 n)}$, so the bound
  $\abs{F(\thetavec)}\le 1$ is adequate in
  conjunction with Lemma~\ref{J0rough}.
\end{proof}

For  disjoint $U, W\subseteq V(G)$ 
define by $\varOmega_{U,W}$
the set of $\thetavec \in  (\Rmodpi)^n$ for which there exists some $x\in\Rmodpi$ and
$\rho$  with $\srho\le \rho \le \brho$ such that the following hold:
\begin{itemize}\itemsep=0pt
	\item[(i)] $\theta_j\in I(x,\srho)$ for at least $4n/5$ components $\theta_j$.
 	\item[(ii)]  $\theta_j \in I(x,\rho+\srho)$ if and only if $j \notin U$.
 	\item[(iii)] $\theta_j\in I(x,\rho+\srho)-I(x,\rho)$ if and only if $j \in W$.
 \end{itemize}

\begin{lemma}
We have
\[
	 (\Rmodpi)^n - \varOmega_0 - \varOmega_1 \subseteq \bigcup_{U,W}  \varOmega_{U,W},
\]
where the union is over all disjoint $U, W\subset V(G)$ with $1\le |U|\le n/5$ and $|W|\le |U|/\log n$. 
\end{lemma}
\begin{proof}
Any $\thetavec \in (\Rmodpi)^n -  \varOmega_1$ is such that 
  at least $4n/5$ of  its components  $\theta_j$  lie in some interval $I(x,\srho)$.
   Suppose it is not covered by  any $\varOmega_{U,W}$.
   For $1 \leq k \leq \log^{2}{n}$, take $\rho = k\srho \le  \brho$ and let $U$
   correspond to the components not in $I(x, \rho +\srho)$.
Since  (iii) cannot hold, we get
\[
	\frac{|\{j\st \theta_j \notin I(x,k\srho)\}|} {|\{j\st \theta_j \notin I(x, (k+1)\srho)\}|}
	= 1 + \frac{|\{j \st \theta_j\in I(x,\rho+\srho)-I(x,\rho)\} |}
	{|\{j\st\theta_j \notin I(x,\rho+\srho)\}|} > 1 +\frac{1}{\log n}.
\]
 Recalling that $|\{j\st\theta_j \notin I(x,\srho)\}| \le n/5$, we can apply this ratio repeatedly starting with $k=1$ to find that 
\[
	|\{j\st\theta_j \notin I(x, \brho)\}| 
	  \le  \dfrac15 n \Bigl(1 +\dfrac{1}{\log n}\Bigr)^{\!-\log^2 n +1}<1.
\]
This implies that $\thetavec \in \varOmega_0$, which completes the proof.
\end{proof}

\begin{lemma}\label{JUWbound}
	For any disjoint $U, W\subset V(G)$ with $|U|\le n/5$ and $|W|\le |U|/\log n$, we have
	\[
		\int_{\varOmega_{U,W}-\varOmega_2} 
		   \abs{F(\thetavec)}\, d \thetavec = e^{-\Omega(|U| \log^4 n)} 
		J_0.
	\]
\end{lemma}
\begin{proof}
Let $X:=V(G)-(U\cup W)$ and define the map 
$\phivec =(\phi_1,\ldots,\phi_n): \varOmega_{U,W} \to \varOmega_0$  as follows.
By the definition of $\varOmega_{U,W}$, for any $\thetavec\in\varOmega_{U,W}$
there is some interval of length at most $\brho$ that contains $\{\theta_j\}_{j\in X}$.
Let $I(z,\xi)$ be the unique shortest such interval. We can ignore parts of
$\varOmega_{U,W}$ that lie in $\varOmega_2$,
which means that we can assume $\xi\ge e^{-\log^3 n}$.

Identifying $\Rmodpi$ with $(z-\frac12\xi,z-\frac12\xi+\pi]$, define
\[
  \phi_j  = \phi_j(\thetavec) :=
    \begin{cases}
	z+\dfrac12\xi - \dfrac{\xi}{\pi-\xi}\(\theta_j-z-\dfrac12\xi\),
			       & \text{ if }  j \in U \cup W;\\
	\theta_j, &  \text{ if } j \in X.
    \end{cases}
\]
For $j\in U\cup W$, $\theta_j\notin I(z,\xi)$ and $\phi_j$ maps the
complementary interval $I(z+\frac12\pi,\pi-\xi)$ linearly onto $I(z,\xi)$
(reversing and contracting with $z\pm\frac12\xi$ fixed).
For $j\in X$, $\theta_j\in I(z,\xi)$ and $\phi_j=\theta_j$.

Thus $\semiabs{\phi_j - \phi_k} \le\semiabs{\theta_j - \theta_k}$ 
for all $j,k$. From Lemma~\ref{l:fapp}, we find that 
\[	
	|f_{jk}(\theta_j - \theta_k)|
\le |f_{jk}(\phi_j - \phi_k)|.
\]
Moreover, for $j\in U$ and $k\in X$, we get that
$\semiabs{\phi_j - \phi_k} \le  \semiabs{\theta_j - \theta_k} -\frac12\srho$.
Observing also that $\semiabs{\phi_j - \phi_k} \le \xi = o(1)$ 
and  using \eqref{fapprox2}, we find that
\[
	\frac
	{|f_{jk}(\theta_j - \theta_k)|}
	{|f_{jk}(\phi_j - \phi_k)|}
	 \le e^{-\Omega(\srho^2)}.
\]
By Assumption A2 of Theorem~\ref{t:bigtheorem}, this
bound applies to at least $h(G)|U| - \varDelta |W| 
\ge (\gamma+o(1))(\varDelta |U|)$ pairs  $jk\in\partial_G U$, thus 
\[
	|F(\thetavec)| = e^{-\Omega(|U| \log^4 n)} |F(\phivec(\thetavec))|.
\]

Note that the map $\phivec$ is injective, since $I(z,\xi)$ can be determined
from $\{\phi_j\}_{j\in X}=\{\theta_j\}_{j\in X}$.
Also, $\phivec$  is analytic except at places where the map from
$\{\theta_j\}_{j\in X}$ to $(z,\xi)$ is non-analytic, which happens
only when two distinct components $\theta_j,\theta_{j'}$ for $j,j'\in X$
lie at the same endpoint of $I(z,\xi)$.
Thus, the points of non-analyticity of $\phivec$ lie on a finite number of hyperplanes,
which contribute nothing to the integral.
To complete the calculation, we need to bound the Jacobian of the transformation
$\phivec$ in the interior of a domain of analyticity.

We have
\[
    \frac{\partial \phi_j}{\partial\theta_k} =
      \begin{cases}
         1, & \text{~ if $j=k\in X$}; \\
          \pm\dfrac{\xi}{\pi-\xi}, & \text{~ if $j=k\notin X$}; \\[-0.5ex]
         0, & \text{~ if $j\ne k$ and either $j\in X$ or $k\notin X$}.
      \end{cases}
\]
Although we have not specified all the entries of the matrix, these entries show that the matrix is triangular, and hence the determinant has absolute value
$\(\frac{\xi}{\pi-\xi}\)^{|U|+|W|}$, which is
$e^{-O(|U| \log^3 n)}$ because  $\xi\ge e^{-\log^3 n}$.
 \end{proof}

\subsection{Proofs of Theorem~\ref{t:bigtheorem} and Lemma~\ref{l:expvar}}\label{s:proofmain}

\begin{proof}[Proof of Theorem~\ref{t:bigtheorem}]
   The number of orientations in terms of the integral $J'$ appears in~\eqref{NGint}.
   That integral restricted to the region $\varOmega_0$ is $J_0$, evaluated in
   Lemma~\ref{t:J0value}.  This gives the expression in Theorem~\ref{t:bigtheorem}
   so it remains to show that the other parts of the integral fit into the
   error terms given there.
   
   The integral in $\varOmega_1\cup\varOmega_2$ is bounded in Lemmas~\ref{S1bound}
   and~\ref{l:Omega2bound}.
   The remaining parts of~$J'$ are bounded by the sum of Lemma~\ref{JUWbound}
   over disjoint $U,W\subset V(G)$ with $1\le |U| \le \frac15n$ and $|W|\le |U|/\log n$.
   The number of choices of $W$ for given $U$ is less than $2^{|U|}$, so the
   total contribution here is
   \[
        J_0 \sum_{t=1}^{n/5} \binom{n}{t} e^{-\Omega(t\log^4 n)}
        \le  \( \( 1 + e^{-\Omega(\log^4 n)}\)^n - 1\)J_0
        = O\( ne^{-\Omega(\log^4 n)}\) J_0,
   \]
   which is easily small enough. 
\end{proof}

\begin{proof}[Proof of Lemma~\ref{l:expvar}]
From Lemma \ref{l:consequences}(c), we know that 
$\lamdiff = O(R)$. Then, applying  Lemma~\ref{l:Ycovmatrix},
we find that 
$\Var f_3(\X)=O\(R^2\varDelta^{-1}n\lognd \)$,
   $\E f_6(\X)=O(\varDelta^{-2}n)$ and
   $\Var f_4(\X)=O\(\varDelta^{-2}n\lognd \)$.
   
 It remains to estimate $\E f_4(\X)=\frac23\sum_{jk\in G}\lamlam (1 - 6\lamlam)\E Y_{jk}^4$,
   which Lemma~\ref{l:Isserlis} shows is equal to
   \[
      2\sum_{jk\in G}\lamlam (1 - 6\lamlam)(\E Y_{jk}^2)^2
       = 2\sum_{jk\in G}\lamlam (1 - 6\lamlam)  \(\sigma_{jj}+\sigma_{kk}-2\sigma_{jk}\)^2,
   \]
   where $(2A)^{-1}=(\sigma_{jk})$. Let $D = \diag(\eta_1,\ldots,\eta_n)$ be the
   diagonal matrix where $\eta_1,\ldots,\eta_n$ are  diagonal elements of $2A$. Using 
  Lemma \ref{l:consequences}(c), we get
   \[
    \lamlam \in \left[\dfrac{1+R}{4+4R+R^2},1 \right]
    \qquad \text{and} \qquad
   	\frac{\eta_j}{d_j} = \frac{4\sum_{k: jk\in G} \lamlam}{d_j} \in \left[\dfrac{4+4R}{4+4R+R^2},1 \right].
   \]  
    Then
   $(2A)^{-1} - D^{-1} = (2A)^{-1} (D-2A) D^{-1}$.  Note that the
   entries of $(D-2A) D^{-1}$ are uniformly $O(\varDelta^{-1})$,
   so the entries of $(2A)^{-1} - D^{-1}$ are uniformly
   $\infnorm{A^{-1}}O(\varDelta^{-1})=
   O\(\varDelta^{-2}\lognd \)$, using
   Lemma~\ref{l:norms}(a).
   Therefore,  for $jk\in G$,
   \[
   \sigma_{jj}+\sigma_{kk}-2\sigma_{jk}\
   = \eta_j^{-1}+\eta_k^{-1} + O\(\varDelta^{-2}\lognd \)
    = d_j^{-1} + d_k^{-1} +O(R^2 \varDelta^{-1}) + O(\varDelta^{-2}\lognd ),
   \]
   where the last equality follows from Lemma \ref{l:consequences}(a). 
   Now it only remains to assemble these parts to obtain the lemma.
\end{proof}

\section{Probability of subdigraph occurrence}\label{ss:eulerian}

Let $H$ be a spanning subgraph of $G$, and let $\vec H$ be an orientation
of $H$ with imbalance sequence $\imbavec'$.  Then
\begin{equation}\label{ratio}
     \frac{N(G\setminus H,\imbavec-\imbavec')}{N(G,\imbavec)}
\end{equation}
is the probability that a uniform random orientation of $G$ with
imbalances $\imbavec$ contains $\vec H$ as a subdigraph.
Consequently, Theorem~\ref{t:bigtheorem} gives this probability
asymptotically provided both the numerator and the denominator satisfy the conditions
of that theorem.  We will not explore this issue further in this paper
except for the case that $\imbavec=\imbavec'=\boldsymbol 0$;
i.e., both orientations are Eulerian.

\begin{theorem}
Let $G$ be a graph with even degrees $d_1,\ldots,d_n$ and
let $H$ be a spanning subgraph of $G$ with even degrees $h_1,\ldots,h_n$.
Define $m=\frac12\sum_{j=1}^n h_j$, and assume that
$\varDelta^{-2}(n+m)\lognd  = o(1)$, where
$\varDelta$ is the maximum degree of~$G$.
Also assume that there is a constant $\gamma>0$ such that $h(G\setminus H)\ge\gamma\varDelta$.  Then, for any fixed Eulerian
orientation $\vec H$  of~$H$, the probability that a random
Eulerian orientation of $G$ includes $\vec H$ is
\[
    2^{-m} \prod_{j=1}^n\,\Bigl(1-\dfrac{h_j}{d_j}\Bigr)^{\!-1/2}
    \exp\Bigl( O\(\varDelta^{-2}(m+n)\lognd \)\Bigr).
\]
\end{theorem}
\begin{proof}
We will evaluate~\eqref{ratio} using Corollary~\ref{c:eulerian}.  Note that
$h(G\setminus H)\ge\gamma\varDelta$ implies $h(G)\ge\gamma\varDelta$,
so assumption~A2 is satisfied by both numerator and denominator. Furthermore, $h(G\setminus H)\ge\gamma\varDelta$ implies that $h_j \leq (1-\gamma)d_j$ for $1\leq j \leq n$.

First, we have
\begin{align*}
  \sum_{jk\in G}  \(d_j^{-1}&+d_k^{-1}\)^2 -
              \sum_{jk\in G\setminus H}  \((d_j-h_j)^{-1}+(d_k-h_k)^{-1}\)^2 \\
 &=  \sum_{jk\in H}  \(d_j^{-1}+d_k^{-1}\)^2
    + \sum_{jk\in G\setminus H} O\( (h_j + h_k)\varDelta^{-3}\) = O(\varDelta^{-2}m).
\end{align*}
Next we consider the ratio $\kappa(G\setminus H)/\kappa(G)$, which equals
the ratio $\card{A'}/\card{A}$, where $A$ is defined as in~\eqref{afdefs}
and $A'$ is the corresponding matrix for $G\setminus H$.
As in the proof of Lemma~\ref{l:expvar}, we have
$A^{-1} = \varLambda+X$, where $\varLambda=\diag(2/d_1,\ldots,2/d_n)$
and $X=(x_{jk})$ with $x_{jk}=O\(\varDelta^{-2}\lognd \)$
for all~$j,k$.
Also $A'=A-\varLambda'+Y$, where $\varLambda'=\diag(h_1/2,\ldots,h_n/2)$
and $Y=(y_{jk})$ with $y_{jk}=\frac12$ for $jk\in H$ and $y_{jk}=0$ otherwise.
We have
\begin{align*}
    \frac{\abs{A'}}{\abs{A}} &= \abs{A^{-1}A'}
    = \abs{I-\varLambda\varLambda'+\varLambda Y-X\varLambda'+XY} \\
    &= \abs{I-\varLambda\varLambda'}\,\abs{I+U}
    = \abs{I+U}\,\prod_{j=1}^n\,\Bigl(1-\dfrac{h_j}{d_j}\Bigr),\\
    &\qquad  \text{~~where $U:=(1-\varLambda\varLambda')^{-1}
               (\varLambda Y-X\varLambda'+XY)$}.
\end{align*}
The Frobenius norm $\frobnorm{U}$ of $U=(u_{jk})$ is defined
by $\frobnorm{U}^2=\sum_{jk}\abs{u_{jk}}^2$.
By subadditivity,
\[
  \frobnorm{U}^2\le \gamma^{-2}\,\(\frobnorm{\varLambda Y}^2
+ \frobnorm{X(\varLambda'-Y)}^2\),
\]
We have $\frobnorm{\varLambda Y}^2=O(\varDelta^{-2}m)$, and
\begin{align*}
   \frobnorm{X(\varLambda'-Y)}^2
    &= \sum_{j,k=1}^n\,\biggl( \dfrac12 x_{jk}h_k - \sum_{t=1}^n x_{jt}y_{tk}\biggr)^{\!2}
     = \sum_{j,k=1}^n\biggl( \sum_{t=1}^n\,
          (x_{jk}y_{tk}-x_{jt}y_{tk})\biggr)^{\!2} \\
    &= \dfrac12 \sum_{j=1}^n \sum_{tk\in H} (x_{jk}-x_{jt})^2
       = O\(\varDelta^{-4}mn\log^2\dfrac{2n}{\varDelta}\)
       = o\( \varDelta^{-2}(m+n)\lognd\),
\end{align*}
where the last equality follows from the theorem assumptions.
Thus, $\frobnorm{U}=o(1)$.
Schur's Inequality ~\cite[p.~50]{Zhan} says that
$\sum_j \abs{\lambda_j}^2\le \frobnorm{U}^2$, where
$\{ \lambda_j\}$ are the eigenvalues of $U$, so
\begin{align*}
   \abs{I+U} &= \exp\biggl( \sum_{j=1}^n \lambda_j
          + O\Bigl( \sum_{j=1}^n \abs{\lambda_j}^2\Bigr)\biggr) \\
             &= \exp\( \tr U + O(\frobnorm{U}^2) \).
\end{align*}
By the definition of $U$ and the above bound on the entries of $X$, $\tr U = O(\varDelta^{-2}m\lognd)$. Thus,
\[
\abs{I+U} = \exp\( O(\varDelta^{-2}(m+n)\lognd )\),
 \]
which completes the proof.
\end{proof}

\begin{corollary}\label{c:hamiltonian}
  Under the conditions of the theorem, if $G$ has $N_H$ hamiltonian
  cycles, then the expected number of directed hamiltonian cycles
  in a random Eulerian orientation of~$G$ is
  \[
      2^{-n+1} N_H \exp\Bigl(\, \sum_{j=1}^n d_j^{-1} +
                            O\(\varDelta^{-2}n\lognd \)\Bigr).
  \]
\end{corollary}

\nicebreak
\section{Appendix}

Here we will collect some technical lemmas that are used in the proof.
This section is self-contained and does not rely on assumptions other
than those stated.

\subsection{Weighted graphs and proof of Lemma~\ref{l:tameness}}
\label{ss:A1}

\begin{lemma}\label{l:generalweights}
		 Let $G$ be a connected graph of maximum degree $\varDelta$. 
		 Suppose each edge $jk \in E(G)$ is assigned a weight $w_{jk} \ge 0$
	  and 
	 \[
	    \bar{w} := \max_{0<s<n}
	 	\frac{\sum_{jk \in \partial_G \{1,\ldots, s\}} w_{jk}}
	 	{|\partial_G \{1,\ldots, s\}|}>0.
	 \]
	 Then, for any $\eta>0$, there exist a 
	 set  of  edges $\calS \in E(G)$  such that
	  \begin{itemize}\itemsep=0pt
	  	\item[(i)]
	 $w_{jk}\le (1+\eta)\bar{w}$ for all $jk \in \calS$;
	 \item[(ii)] the intervals of real numbers $\{ [j,k] \st jk\in \calS, j<k\}$ cover $[1,n]$;
	 \item[(iii)] $\displaystyle|\calS| \le 4 
	   + \frac{2\log \Bigl(\frac{n (1+\eta)}{2 \eta h(G) } \Bigr)}
	             {\log \Bigl(1+ \frac{\eta h(G)}{(1+\eta)\varDelta}\Bigr) }$. 
	 \end{itemize}
\end{lemma}
\begin{proof}
	Consider the spanning subgraph $H$ of $G$ constructed as follows: each edge $jk\in G$ 
	is present in $H$ if and only if $w_{jk}\le (1+\eta)\bar{w}$.  Note that, for any 
	$0\le s <n$, we have 
	\[
	\bar{w}\,
	\abs{\partial_G \{1,\ldots, s\}} 	\ge  {\sum_{jk \in \partial_G \{1,\ldots, s\}}\!\! w_{jk}}\ge 
		(1+\eta)\bar{w}\,\( |\partial_G \{1,\ldots, s\}| -  |\partial_H \{1,\ldots, s\}| \).
	\]
	Observing also $\partial_G \{1,\ldots, s\} = \partial_G \{s+1,\ldots, n\}$, we get 
	\begin{equation}\label{H_exp}
		|\partial_H \{1,\ldots, s\}| \ge \frac{\eta}{1+\eta}|\partial_G \{1,\ldots, s\}|
		\ge \frac{\eta}{1+\eta}  h(G)  \min\{s,n-s\}.
	\end{equation}
	
	Now we will construct~$\calS$.  By applying equation~\eqref{H_exp}
	for $s=1$, we can start with $\calS=\{1k\}$, where $1k\in H$ and
	$k\ge 1+\frac{\eta h(G)}{1+\eta}$.
	From here we proceed recursively.  Suppose we have edges covering
	$[1,\ell]$ (in the sense of (ii)), where $\ell<n/2$.
	Applying~\eqref{H_exp} to $\{1,\ldots,\ell\}$ and recalling that all vertices
	have degree at most $\varDelta$, there must be at least
	$\frac{\eta  h(G)}{(1+\eta)\varDelta}\ell$ vertices  in $\{\ell+1,\ldots, n\}$
	that in $H$ have neighbours in $\{1,\ldots,\ell \}$.
	So there is some $k\ge \ell\(1+\frac{\eta  h(G)}{(1+\eta)\varDelta}\)$ such
	that $jk\in H$ for some $j\le\ell$.
	Adding this edge to $\calS$ means that we have covered $[1,k]$.
	Continuing in this manner, we will have covered $[1,n/2]$ while
	$\calS$ has at most
	\[
	    1 + \left\lceil \frac{\log \Bigl(\frac{n (1+\eta)}{2 \eta h(G) } \Bigr)}
	             {\log \Bigl(1+ \frac{\eta h(G)}{(1+\eta)\varDelta}\Bigr)} \right \rceil
	\]
	edges from $H$.
	Finally, repeat the process starting at vertex~$n$ to find a
	similar set of edges that cover $[n/2,n]$.  This completes the proof.
\end{proof}

\medskip
\begin{proof}[Proof of Lemma~\ref{l:tameness}]
 Without loss of generality we may assume  $r_1\ge \ldots \ge r_n$.  
  We employ Lemma \ref{l:generalweights}, where for any $jk \in G$
  we take $j< k$ and define $w_{jk}$ by 
  \[
  w_{jk} := \frac{r_j - r_k}{r_j + r_k} = \lamdiff \ge 0.
  \]  
  Note that $\sum_{jk \in \partial_G\{1,\ldots,s\}}
  w_{jk} = \sum_{j=1}^s \imba_j$. Thus, by assumptions, we get $\bar{w} \le 1-\delta$. 
  Take  $\eta = \delta$ and consider the set  $\calS$ constructed in Lemma \ref{l:generalweights}.
  For $w_{jk} \le (1+\eta) \bar{w}$, we have 
  \begin{equation}\label{bjminusbk}
      \Abs{\log \dfrac{r_j}{r_k} } = \log\Bigl(\frac{1+w_{jk}}{1-w_{jk}}\Bigr) 
      \le \log(2\delta^{-2}-1)\le 4\,\log \dfrac{1}{\delta}\,.
  \end{equation}
  Also, observe that
  \[
  	\card \calS  \le 4 +  2\,\log  \Bigl(\frac{n (1+\delta)}{2 \delta h(G)} \Bigr)
	\Big/
  	\log\Bigl(1+ \frac{\delta h(G)}{(1+\eta)\varDelta}\Bigr).
  \]
 By~\cite[Thm.~2.2]{Mohar1989}, for $n \ge 10$ we have  
 $h(G)\le \frac{\lceil n/2\rceil}{n-1}\varDelta\le \frac35\varDelta$ and also $h(G)\le h(K_n) \le \frac{6}{11}n$.
  Now we can calculate
  \[
       \card \calS \le (4 A_1 + 2 A_2 A_3) \frac{\varDelta}{\delta h(G)} \log \frac{n}{\delta h(G)},
 \]
 where
 \begin{align*}
   A_1 &:= \frac{\delta h(G)}{\varDelta}\Big/\log\frac{n}{\delta h(G)}
                 \le \dfrac35\bigm/\log\dfrac{11}{6}, \\   
   A_2 &:= \log \frac{(1+\delta)n}{2\delta h(G)}\Big/
          \log \frac{n}{\delta h(G)} \le 1, \text{ and}\\
   A_3 &:= \frac{\delta h(G)}{\varDelta} \Big/ 
      \log\biggl(1 + \frac{\delta h(G)}{(1+\delta)\varDelta}\biggr)
      \le \dfrac35\bigm/\log\dfrac{13}{10}.
 \end{align*}
 In each case the bounds on the right hand side follow from the fact that the supremum occurs as $\delta\to 1$ and $h(G)$ has the
greatest allowed value.
 
Then, from property (ii) of Lemma~\ref{l:generalweights}
    and~\eqref{bjminusbk}, we find that
  \[
  	\Abs{\log \dfrac{r_1}{r_n} } \le \sum_{jk \in S}\, \Abs{\log \dfrac{r_j}{r_k} }
	\le 4\, \card \calS\log\dfrac{1}{\delta}, 
  \]
  where $jk \in \calS$ in the sum is ordered as $j<k$. The result follows
  on applying the above numerical bounds.
\end{proof}

\nicebreak
\subsection{Matrices and norms}

\begin{lemma}\label{app_matrix2}
Let $L$ be a symmetric matrix with nonpositive off-diagonal elements and
zero row sums. Suppose the eigenvalues of $L$ are $0=\mu_1<\mu_2\le\cdots\le\mu_n$.
For any real $\alpha$, define the matrix $L_\dagger^\alpha$ by
$L_\dagger^\alpha\xvec=\mu_2^\alpha\vvec_2+\cdots+\mu_n^\alpha\vvec_n$,
where $\xvec=\vvec_1+\cdots+\vvec_n$ is the decomposition of $\xvec$ as
a sum of eigenvectors of $L$ (numbered consistently with the eigenvalues).
Then
\[
  \infnorm{L_\dagger^\alpha} \le (2\infnorm{L})^\alpha
  \sum_{k=0}^\infty\; \biggl|\binom{\alpha}{k}\biggr|
  \min\biggl\{ 2, \sqrt n\biggl(1 - \frac{\mu_2}{2\infnorm{L}}\biggr)^{\!k}\,
      \biggr\}. 
\]
\end{lemma}
\begin{proof}
Let $X:=I - (2\infnorm{L})^{-1}L$.
The eigenvalues of $X$ are $1=\nu_1>\nu_2\ge\cdots\ge\nu_n$, where
$\nu_j=1-(2\infnorm{L})^{-1}\mu_j$ for each~$j$.
Since $\abs{\nu_j}<1$ for $2\le j\le n$, we have
\begin{align*}
   L_\dagger^\alpha\xvec &= (2\infnorm{L})^\alpha
            \sum_{j=2}^n\, (1-\nu_j)^\alpha\vvec_j \\
      &= (2\infnorm{L})^\alpha
         \sum_{k=0}^\infty\, (-1)^k\binom{\alpha}{k}
         \sum_{j=2}^n \nu_j^k\vvec_j \\
      &= (2\infnorm{L})^\alpha
         \sum_{k=0}^\infty\, (-1)^k\binom{\alpha}{k}
         \, X^k(\vvec_2+\cdots+\vvec_n) \\
       &= (2\infnorm{L})^\alpha \sum_{k=0}^\infty\,
          (-1)^k\binom{\alpha}{k} X^k(I-\dfrac1nJ)\,\xvec,
\end{align*}
where we have used the fact that $\vvec_1=\frac1nJ\xvec$.
We will now find two different bounds on $\infnorm{X^k(I-\frac1nJ)}$.
First note that $\infnorm{X}=1$ so $\infnorm{X^k(I-\frac1nJ)}\le \infnorm{I-\frac1nJ}<2$.
Second, the maximum eigenvalue of $X^k(I-\frac1nJ)$ is $\nu_2^k$, so
$\infnorm{X^k(I-\frac1nJ)}\le \sqrt n\, \twonorm{X^k(I-\frac1nJ)} \le \sqrt n\,\nu_2^k$.
Combining these two bounds completes the proof.
\end{proof}

\begin{corollary}\label{c:Apowers}
  For $c>0$, consider the positive-definite matrix $A:=\frac cn J + L$,
  where $L$ satisfies the conditions of Lemma~\ref{app_matrix2} with
  $\nu_2=1-(2\infnorm{L})^{-1}\mu_2$.
  Then, for any real $\alpha\ge -1$, the positive-definite power $A^\alpha$
  satisfies
  \[
    \infnorm{A^\alpha} \le c^\alpha 
    + (2\infnorm{L})^\alpha \biggl( 2\sum_{k=0}^{N-1}\biggl|\binom{\alpha}{k}\biggr| 
    + n^{-1/2}/(1-\nu_2)\biggr),
  \]
  where $N=\lceil \abs\alpha + \log_{\nu_2} n^{-1}\rceil$.
\end{corollary}
\begin{proof}
  Since $A$ has the same eigenvectors as $L$, and the same eigenvalues
  except that~$0$ has been replaced by~$c$, we have
  \[
      A^\alpha = \dfrac{c^\alpha}{n}J + L_\dagger^\alpha.
  \]
  Now we can apply the Lemma in the obvious way, using
  $\sqrt n\,\nu_2^k\le n^{-1/2}\nu_2^{N-k}$ for $k\ge N$
  and $\bigl|\binom\alpha k\bigr|\le 1$ for $\alpha\ge -1$
  and $k\ge\abs\alpha$.
\end{proof}

In some cases we can improve on Corollary~\ref{c:Apowers}.
We will only use a bound on $\infnorm{A^{-1}}$.
 
\begin{lemma}\label{l:inversenorm}
  Let $G$ be a connected graph of maximum degree $\varDelta$. 
  Let $L=(\ell_{jk})$ be a symmetric matrix with zero row sums
  such that, for $j\ne k$, $\ell_{jk}=0$ if $jk\notin G$ and
  $\ell_{jk}<-\ellmin $ if $jk\in G$, for some $\ellmin >0$.
  Define $A:=\frac cnJ+L$ for $c>0$.
  Then, if $n\ge 10$,
  \[
       \infnorm{A^{-1}} \le c^{-1} 
           + \frac{18\varDelta}{\ellmin \,h(G)^2} \log \frac{n}{h(G)}.
  \]
\end{lemma}
 \begin{proof}
  As in Corollary~\ref{c:Apowers}, we have
  $\infnorm{A^{-1}}\le c^{-1}+\infnorm{L^{-1}_\dagger}$, where
  $L^{-1}_\dagger$ is defined in Lemma~\ref{app_matrix2}.
  Moreover,
  \[
       \infnorm{L^{-1}_\dagger} =
          \max_{\xvec} \frac{\infnorm\xvec}{\infnorm{L\xvec}},
  \]
  where the maximum is taken over $\xvec\ne\boldsymbol 0$ such that
  $x_1+\cdots+x_n=0$.
  Permuting $L$ if necessary, we can assume that the maximum occurs
  for $\xvec$ with $ x_1 \ge \cdots \ge x_n$.
          Let $\yvec =(y_1,\ldots,y_n):= L\xvec$, and for
           $jk \in E(G)$ and $j<k$, put $w_{jk} := -\ell_{jk} (x_j-x_k)$.
           Observe that, for $1\le j\le n$,
           \[
                 y_j = \sum_{k:jk\in G} \!\ell_{jk} x_k - 
                          x_j \!\sum_{k:jk\in G} \!\ell_{jk}
                        = - \!\sum_{k:jk\in G}\! \ell_{jk}(x_j-x_k),
           \]
           from which it follows that for $1\le s\le n$,
           \[
               \sum_{j=1}^s y_j = \sum_{jk\in\partial_{G\{1,\ldots,s\}}} w_{jk},
           \]
           taking $j<k$ in the sum.
           Since $JL=0$ we have $ \sum_{j=1}^s y_j =-\sum_{j=s+1}^n y_j$,
           so by the definition of $h(G)$ we have
           \[
               \sum_{jk\in\partial_{G\{1,\ldots,s\}}} \!\!w_{jk} \le
                   \min\{s,n-s\} \infnorm\yvec
                 \le \infnorm\yvec\frac {|\partial_G\{1,\ldots,s\}|} {h(G)}.
           \]
          Thus, defining $\bar w$ as in Lemma~\ref{l:generalweights},
          we have  $\infnorm{L \xvec} \ge h(G) \bar{w}$.
          Since $x_1+\cdots+x_n=0$, we have $x_1-x_n\ge \infnorm\xvec$.
          Taking the set $\calS$ of edges guaranteed by Lemma \ref{l:generalweights}
          with $\eta =1$, we find that
          \begin{align*}
          	\infnorm\xvec &\le 
	             x_1 -x_n \le \sum_{jk \in S} \,(x_j - x_k) 
	             \le \frac{2 \bar{w}}{\ellmin \,} |\calS|
          	\\
          	&\le  \frac{2\, \infnorm{L \xvec}}{\ellmin \, h(G)} 
           \Biggl( 4 + \frac{2\log \frac{n}{ h(G) } }
	                           {\log \Bigl(1+ \frac{h(G)}{2\varDelta}\Bigr) }\Biggr).
          \end{align*}
          To complete the numerical bound, continue as in the proof
          of Lemma~\ref{l:tameness}; we omit the
          uninteresting details. 
 \end{proof}

\nicebreak
\subsection{Short paths}

\begin{lemma}\label{l:paths} Let $G$ be a graph of maximum degree $\varDelta$. Assume also that
$h(G)\ge \gamma \varDelta$ for some $\gamma>0$.  
 For any two disjoint sets of vertices  $U_1, U_2$, denote
\[
	\ell(U_1,U_2) = 2+2\log_{1+\gamma/2} 
	  \biggl(\frac{|V(G)|}{  \min\{|U_1|,|U_2|\}+  \gamma  \varDelta/2}\biggr).
\]
Then, there exist at least $\gamma \varDelta \dfrac{ \min\{|U_1|,|U_2|\}}{2\ell(U_1,U_2)} $
 pairwise edge-disjoint paths in $G$
  with  one end in~$U_1$ and the other end in $U_2$ of lengths  
  bounded above by $\ell(U_1,U_2)$. 
\end{lemma}

\begin{proof}
	 Let $n$ be the number of vertices of $G$. 
	Denote $u:=\min\{|U_1|,|U_2|\}$.
	Without loss of generality we may assume that $|U_1|=|U_2| = u$ because we can
	always remove some vertices from the larger set.
	We call a path \textit{short} if it has length at most $\ell(U_1,U_2)$.
	For a subgraph $H$  denote 
	\[
		h_u(H):=\min_{u\le |U|\le\frac{n}{2}}	 \frac{|\partial_H \,U|}{|U|}.  
	\]
	
	\medskip
	
	Starting from $H=G$, we construct the required set of short paths by repeating the following procedure.
	\begin{enumerate}\itemsep=0pt
		\item[(1)] If   $h_u(H) \ge \gamma \varDelta/2$  then do (2), otherwise STOP.
		\item[(2)] Find a path $P$ in $H$ of length  at most 
		\[ 
		 2+ 2  \min\biggl\{\log_{1+\gamma/2} \biggl(\frac{n}{ 2u}\biggr), \log_{1+\gamma/2}
		  \biggl(\frac{n}{ \gamma \varDelta}\biggr) \biggr\} \le
		  \ell(U_1,U_2).
		 \]  
         Add $P$ to the set of constructed paths.
         Delete the edges of  $P$ from $H$ and repeat from~(1).		 
	 \end{enumerate}

	  Suppose,  we found fewer than 
	  $\dfrac{\gamma \varDelta u}{2\ell(U_1,U_2)}$ paths by the procedure above,  
	  so that, in particular, we deleted less than $\gamma \varDelta u/2$ edges. Therefore, for any  $U$ such that 
		$u\le |U|\le n/2$, 
		\[
			\frac{|\partial_H\, U|}{|U|} \ge 
			h(G) - \frac{\gamma \varDelta u}{2|U|}  \ge \gamma \varDelta/2.
		\]
		Thus, $h_u(H) \ge \gamma \varDelta/2$.
		
		\medskip
		
	  Now, we explain why (1) implies the existence of a short path from $U_1$ to $U_2$.
	  Indeed, for  $u\le |U| \le n/2$, we have 
	 \[ 
	  |N_H(U)|  \ge   \frac{|\partial_H \,U|}{|U|}  \ge h_u(H) \ge  \gamma \varDelta/2,
	  \]
	  	  	  where $N_H(U)$ denotes the neighbourhood of $U$ in $H$.
	  Since the number of edges from any vertex of $U$  to $N_H(U)$ is bounded by $\varDelta$, we get that
	 \[ 
	    |U \cup  N_H(U)| \ge  (1 + \gamma/2) |U|.
	  \]
	  Therefore,  we can reach more than $n/2$ vertices starting  from $U_1$ (or from $U_2$) by 
	  paths of  length at most  
	  $\log_{1+\gamma/2} \(\dfrac{n}{ 2u}\)$.
	  Alternatively, since  $|N(U_1)|\ge \gamma \varDelta/2$,
	  we can reach more than $n/2$ vertices starting from $N(U_1)$ by 
	  paths of length at most  
	 $\log_{1+\gamma/2} \(\dfrac{n}{\gamma\varDelta}\)$
	 (and the same holds for $U_2$).
	 Therefore, we can find a vertex which is not too distant from both
	 $U_1$ and $U_2$ and construct the required short path $P$
	 
\medskip
	 
Our procedure will stop at some moment since  $G$ is finite. 
As  shown above, this can only happen after we found at least
$\dfrac{\gamma \varDelta u}{2\ell(U_1,U_2)}$ edge-disjoint short paths from
$U_1$ to $U_2$. This completes the proof.
\end{proof}

\subsection{Integration theorem}\label{A:integ}

For the reader's convenience, we quote \cite[Lemma~4.6]{Mother}
and \cite[Theorem~4.4]{Mother}
with very minor changes to match the notations of this paper.

If $T:\Reals^n\to \Reals^n$ is a linear operator, let 
$\ker T := \{\xvec\in \Reals^n \st T\xvec = \boldsymbol{0} \}$. 

\begin{lemma}\label{LemmaQW}
  Let $S,W:\Reals^n\to \Reals^n$ be linear operators 
  such that $\ker S \cap \ker W = \{\boldsymbol{0}\}$ and ${\rm span}(\ker S, \ker W) = \Reals^n$. 
  Let $\nperp$  denote the dimension of\/ $\ker S$.
  Suppose $\varOmega \subseteq \Reals^n$ and
  $F:\varOmega\cap S(\Reals^n) \to\Complexes$.
  For any $\rho>0$, define
  \[
   \varOmega^\rho := \bigl\lbrace \xvec\in\Reals^n \st
     S\xvec\in \varOmega \text{~and~}
     W\xvec\in U_n(\rho) \bigr\rbrace.
  \]
   Then, if the integrals exist,
   \[
    \int_{\varOmega \cap S(\Reals^n)} F(\yvec)\,d\yvec
    = (1 - K)^{-1}\,\pi^{-\nperp/2} \,\Abs{S\trans\! S + W\trans W}^{1/2}
     \int_{\varOmega^{\rho}} F(S\xvec)\, e^{-\xvec\trans\! W\trans W \xvec}
        \,d\xvec,
   \]
   where 
   \[
   0\le K < \min\{1,n e^{-\rho^2/\kappa^2}\}, \ \  
   \kappa := \sup_{W\xvec \neq 0}  \frac{\norm{W\xvec}_\infty} {\norm{W\xvec}_2} \leq 1.
   \]
      Moreover, if  $U_n(\rho_1) \subseteq \varOmega \subseteq  U_n(\rho_2)$ for some $\rho_2 \geq \rho_1 >0$ then
  \[
      U_n\biggl(\min\Bigl\{\frac{\rho_1}{\norm{S}_\infty},
          \frac{\rho}{\norm{W}_\infty} \Bigr\}\biggr)
        \subseteq \varOmega^{\rho} \subseteq
      U_n\( \norm{P}_\infty\, \rho_2 +  \norm{R}_\infty\, \rho \)
  \]	
   for any linear operators $P,R : \Reals^n \to \Reals^n$ such that $PS + RW$ is equal to the identity operator on $\Reals^n$.  
\end{lemma}

%
For a domain $\varOmega \subseteq\Reals^n$ and a 
twice continuously-differentiable function $q:\varOmega\to\Complexes$,
define
\[
   \mathcal{H}(q,\varOmega) = ( h_{jk}), \text{ where }
   h_{jk} := \sup_{\xvec\in\varOmega}\,
   \Bigl|\dfrac{\partial^2 q(\xvec)}{\partial x_j\,\partial x_k}\Bigr|.
\]
For a complex number $z$ we denote by $\Re(z)$ and $\Im(z)$ the real and imaginary parts, respectively.

\begin{theorem}\label{t:gauss4pt}
  Let $c_1,c_2,c_3,\eps,\hat\rho_1,\hat\rho_2,\phi_1,\phi_2$ be
  nonnegative real constants with $c_1,\eps>0$.
 Let $A$ be an $n\times n$ positive-definite symmetric real matrix
 and let $T$ be a real matrix such that $T\trans\! AT=I$.
 
 Let $\varOmega$ be a measurable set such that
 $U_n(\hat\rho_1)\subseteq T^{-1}(\varOmega)\subseteq U_n(\hat\rho_2)$,
   and let
   $f: \Reals^n\to\Complexes$,
   $g: \Reals^n\to\Reals$ and $\rem:\varOmega\to\Complexes$
   be twice continuously-differentiable functions.
 We make the following assumptions.
 \nicebreak
   \begin{itemize}\itemsep=0pt
     \item[(a)] $c_1(\log n)^{1/2+\eps}\le\hat\rho_1\le\hat\rho_2$.
       
     \item[(b)] For $\xvec\in T(U_n(\hat\rho_1))$,\\
        $2\hat\rho_1\,\onenorm{T}\,\abs{\partial f(\xvec)/\partial x_j}
         \le \phi_1 n^{-1/3}\le\tfrac23$ for $1\le j\le n$ and\\
         $4\hat\rho_1^2\,\onenorm{T}\,\infnorm{T}\,
         \infnorm{\mathcal{H}(f,T(U_n(\hat\rho_1)))}
         \le \phi_1 n^{-1/3}$.
                  
     \item[(c)] For $\xvec\in\varOmega$, $\Re f(\xvec) \le g(\xvec)$.
        For $\xvec\in T(U_n(\hat\rho_2))$, either\\
       (i) $2\hat\rho_2\,\onenorm{T}\,\abs{\partial g(\xvec)/\partial x_j}\le
        (2\phi_2)^{3/2} n^{-1/2}$ for $1\le j\le n$, or\\
       (ii) $2\hat\rho_2\,\onenorm{T}\,\abs{\partial g(\xvec)/\partial x_j}
         \le \phi_2 n^{-1/3}$ for $1\le j\le n$ and\\
         \hspace*{1.7em}$4\hat\rho_2^2\,\onenorm{T}\,\infnorm{T}\,
          \infnorm{\mathcal{H}(g,T(U_n(\hat\rho_2)))}
         \le  \phi_2 n^{-1/3}$.
     
     \item[(d)] $\abs{f(\xvec)},\abs{g(\xvec)} \le n^{c_3}
                    e^{c_2\xvec\trans\! A\xvec/n}$ for $\xvec\in\Reals^n$.
    \end{itemize}
   Let $\X$ be a random variable with the normal density
    $\pi^{-n/2} \abs{A}^{1/2} e^{-\xvec\trans\!A\xvec}$.
     Then, provided $\E\,(f(\X)-\E f(\X))^2$ and $\Var g(\X)$ are finite
     and $\rem$ is bounded in~$\varOmega$,
     \[
        \int_\varOmega e^{-\xvec\trans\!A\xvec + f(\xvec)+\rem(\xvec)}\,d\xvec
        = (1+K) \pi^{n/2}\abs{A}^{-1/2} e^{\E f(\X)+\frac12\E\,(f(\X)-\E f(\X))^2},
     \]
   where, for some constant $C$ depending only on $c_1,c_2,c_3,\eps$,
   \begin{align*}
      \abs{K} &\le C\, e^{\frac12\Var\Im f(\X)}\,\Bigl( e^{\phi_1^3+e^{-\hat\rho_1^2/2}}-1
        \\
      &{\qquad}+ \(2e^{\phi_2^3+e^{-\hat\rho_1^2/2}}-2
        + \sup_{\xvec\in\varOmega}\,\abs{e^{\rem(\xvec)}-1} \)\,
           e^{\E(g(\X)-\Re f(\X))+\frac12(\Var g(\X)-\Var\Re f(\X))}\Bigr).
   \end{align*}
   In particular, if $n\ge (1+2c_2)^2$ and
   $\hat\rho_1^2 \ge 15 + 4c_2 + (3+8c_3)\log n$,
    we can take~$C=1$.
\end{theorem}  

\nicebreak

\end{document}